\documentclass[11pt,reqno,twoside]{amsart}
\usepackage{graphicx}
\usepackage{epstopdf}
\usepackage[asymmetric,top=3.5cm,bottom=4.3cm,left=3.1cm,right=3.1cm]{geometry}
\usepackage{bm}

\usepackage{booktabs} % for much better looking tables
\usepackage{array} % for better arrays (eg matrices) in maths
\usepackage{paralist} % very flexible & customisable lists (eg. enumerate/itemize, etc.)
\usepackage{verbatim} % adds environment for commenting out blocks of text & for better verbatim
\usepackage{subfig} % make it possible to include more than one captioned figure/table in a single float
\usepackage{tabularx}
\usepackage{amsmath,amsfonts,amsthm,mathrsfs,amssymb,cite}
\usepackage[usenames]{color}

\usepackage{stmaryrd}
\usepackage{amsmath,amssymb,amsfonts}
\usepackage{times}
\usepackage[T1]{fontenc}
\usepackage{graphicx}
\usepackage[english]{babel}
\usepackage{amsmath}
\usepackage{amsthm}
\usepackage{amssymb}
\usepackage{enumerate}
\usepackage{xspace}
\usepackage{euscript}
\usepackage{graphicx}
\usepackage{amscd}
\usepackage{epsfig}
\usepackage{tabularx}
\usepackage{color}

\newtheorem{thm}{Theorem}[section]

\newtheorem{lem}{Lemma}[section]

\theoremstyle{definition}

\theoremstyle{remark}

\newtheorem{rem}{Remark}[section]
\numberwithin{equation}{section}

\title[Isotropic cloaking and transmission eigenvalues]{On isotropic cloaking and interior transmission eigenvalue problems}

\author{Xia Ji}
\address{LSEC, Institute of Computational Mathematics, Chinese Academy of Sciences, Beijing, 100190, P.~R.~China.}
\email{jixia@lsec.cc.ac.cn}

\author{Hongyu Liu}
\address{Department of Mathematics, Hong Kong Baptist University, Kowloon Tong, Hong Kong SAR, and HKBU Institute of Research and Continuing Education, Virtual University Park, Shenzhen, P. R. China. }
\email{hongyu.liuip@gmail.com}

\date{April 19, 2016}                                           % Activate to display a given date or no date

\begin{document}

\begin{abstract}

This paper is concerned with the invisibility cloaking in acoustic wave scattering from a new perspective. We are especially interested in achieving the invisibility cloaking by completely regular and isotropic mediums. It is shown that an interior transmission eigenvalue problem arises in our study, which is the one considered theoretically in \cite{CCH}. Based on such an observation, we propose a cloaking scheme that takes a three-layer structure including a cloaked region, a lossy layer and a cloaking shell. The target medium in the cloaked region can be arbitrary but regular, whereas the mediums in the lossy layer and the cloaking shell are both regular and isotropic. We establish that if a certain non-transparency condition is satisfied, then there exists an infinite set of incident waves such that the cloaking device is nearly-invisible under the corresponding wave interrogation. The set of waves is generated from the Herglotz-approximation of the associated interior transmission eigenfunctions. We provide both theoretical and numerical justifications.

\medskip

 \smallskip


\end{abstract}

\maketitle

%\tableofcontents

\section{Introduction}

\subsection{Mathematical formulation and motivation}

In this paper, we are concerned with the scattering of acoustic waves in the time-harmonic regime. Let $\Omega\subset\mathbb{R}^N$, $N=2, 3$, be a bounded Lipschitz domain with a connected complement $\mathbb{R}^N\backslash\overline{\Omega}$. Let $\sigma(x)=(\sigma^{jl}(x))_{j,l=1}^N$, $x=(x^j)_{j=1}^N\in\mathbb{R}^N$. It is assumed that $\sigma^{jl}=\sigma^{lj}$, $1\leq j, l\leq N$, are real-valued and measurable functions such that
\begin{equation}\label{eq:reg1}
 \exists\ \lambda \in (0, 1),\  \forall\xi=(\xi_j)_{j=1}^N\in\mathbb{R}^N,\  \forall x\in\mathbb{R}^N,\ \lambda \|\xi\|^2\leq \sum_{j,l=1}^N \sigma^{jl}\xi_j\xi_l\leq \lambda^{-1} \|\xi\|^2,
 \end{equation}
 and
 \[
 \forall x\in \mathbb{R}^N\backslash\overline{\Omega},\quad \sigma^{jl}=\delta_{jl},
 \]
 where $\delta_{jl}$ is the Kronecker delta function, $1\leq j, l\leq N$. We also let $n(x)=\Re n(x)+\mathrm{i}\Im n(x)\in L^\infty(\mathbb{R}^N)$ be a complex-valued function such that
\begin{equation}\label{eq:reg2}
 \exists\, n_0>0,\ \ \forall x\in\mathbb{R}^N,\ \mbox{$\Re n(x)>n_0$\ \  and\ \ $\Im n(x)\geq 0$.}
\end{equation}
 It is further assumed that $n(x)\equiv 1$ for $x\in\mathbb{R}^N\backslash\overline{\Omega}$. $(\mathbb{R}^N; \sigma, n)$ denotes an acoustic medium in $\mathbb{R}^N$ with its inhomogeneity supported in $\Omega$. $\sigma^{-1}$ signifies the density tensor of the acoustic medium, whereas $\Re n$ and $\Im n$ are, respectively, associated with the modulus and the loss of the acoustic material parameters. We remark that \eqref{eq:reg1} and \eqref{eq:reg2} are the {\it physical conditions} satisfied by a regular medium. In what follows, $\sigma$ and $n$ are said to be {\it uniformly elliptic} with constants $\lambda$ and $n_0$, if they respectively, satisfy \eqref{eq:reg1} and \eqref{eq:reg2}. The inhomogeneity of the acoustic medium $(\Omega; \sigma, n)$ is located in an isotropic and homogeneous background/matrix medium whose material parameters are normalised to be $\sigma=\mathbf{I}_{N}$ and $n=1$, where $\mathbf{I}_{N}$ denotes the $N\times N$ identity matrix. $(\Omega; \sigma, n)$ is referred to as a {\it scatterer} in the following. The scatterer $(\Omega; \sigma, n)$ is said to be {\it isotropic} if $\sigma^{jl}(x)=\gamma(x)\delta_{jl}$, $1\leq j, l\leq N$, for a scalar function $\gamma\in L^\infty(\Omega)$, otherwise it is called {\it anisotropic}.

 Next, we consider the time-harmonic acoustic wave propagation in the space $(\mathbb{R}^N;\sigma, n)$. Let $u^i(x)$, $x\in\mathbb{R}^N$ be an entire solution to the following Helmholtz equation
 \begin{equation}\label{eq:Helm1}
 \Delta u^i+\kappa^2 u^i=0\quad\mbox{in}\ \ \mathbb{R}^N,
 \end{equation}
 where $\kappa\in\mathbb{R}_+$ denotes a wavenumber. The propagation of the acoustic wave $u^i$ will be interrupted due to the presence of the inhomogeneity $(\Omega; \sigma, n)$, and this leads to the so-called {\it scattering}. We let $u^s$ denote the perturbed/scattered wave field and $u=u^i+u^s$ denote the total wave field. The total wave field $u\in H_{loc}^1(\mathbb{R}^N)$ is governed by the following Helmholtz system
\begin{equation}\label{eq:Helm2}
\begin{cases}
& \displaystyle{\sum_{j,l=1}^N\frac{\partial }{\partial x_j}\left(\sigma^{jl}(x)\frac{\partial u}{\partial x_l}\right)(x)+\kappa^2 n(x) u(x)=0}\qquad\mbox{for}\ \ x\in\mathbb{R}^N,\smallskip\\
& \mbox{$u-u^i$ satisfies the Sommerfeld radiation condition}.
\end{cases}
\end{equation}
In \eqref{eq:Helm2}, a function $w(x)=u(x)-u^i(x)$ is said to satisfy the Sommerfeld radiation condition if the following limit
\begin{equation}\label{eq:sommerfeld}
\lim_{\|x\|\rightarrow+\infty}\|x\|^{(N-1)/2}\left\{\frac{\partial w(x)}{\partial \|x\|}-\mathrm{i}\kappa w(x)\right\}=0,
\end{equation}
holds uniformly for $\hat{x}:=x/\|x\|$, $x\in\mathbb{R}^N\backslash\{0\}$. It is known that $u$ admits the following asymptotic expansion as $\|x\|\rightarrow+\infty$ (cf. \cite{ColKre,Mcl,Ned}),
\begin{equation}\label{eq:farfield}
u(x)=u^i(x)+\frac{e^{\mathrm{i}\kappa\|x\|}}{\|x\|^{(N-1)/2}}\, u^\infty(\hat x)+\mathcal{O}\left(\frac{1}{\|x\|^{(N+1)/2}}\right),
\end{equation}
which holds uniformly in $\hat x\in\mathbb{S}^{N-1}$. $u^\infty(\hat x)=u^\infty(\hat x; u^i, (\Omega; \sigma, n))$ is known as the {\it far-field pattern} or the {\it scattering amplitude}.

An important inverse scattering problem arising from practical application is to recover $(\Omega; \sigma, n)$ by knowledge of the corresponding far-field pattern $u^\infty(\hat x)$. $(\Omega; \sigma, n)$ is said to be {\it invisible} under the wave interrogation by $u^i$ if $u^\infty(\hat x; u^i, (\Omega; \sigma, n))\equiv 0$, which corresponds to the nonidentification in the inverse scattering problem mentioned above. In the current work, we shall consider the cloaking technique in achieving the invisibility. Let $D\Subset\Omega$ be a bounded Lipschitz domain. Consider a cloaking device of the following form
\begin{equation}\label{eq:cloak}
(\Omega; \sigma, n)=(D;\sigma_a, n_a)\wedge (\Omega\backslash\overline{D}; \sigma_c, n_c).
\end{equation}
In \eqref{eq:cloak}, $(D; \sigma_a, n_a)$ denotes the target object being cloaked, and $(\Omega\backslash\overline{D}; \sigma_c, n_c)$ denotes the cloaking shell. For a practical cloaking device of the form \eqref{eq:cloak}, there are several crucial ingredients that one should incorporate into the design:

\medskip

\noindent(i).~The target object $(D; \sigma_a, n_a)$ can be allowed to be arbitrary (but regular). That is, the cloaking device should not be object-dependent. In what follows, this issue shall be referred to as the target independence for a cloaking device.\smallskip

\noindent(ii).~The cloaking medium $(\Omega\backslash\overline{D}; \sigma_c, n_c)$ should be feasible for construction and fabrication. Indeed, it would be the most practically feasible if $(\Omega\backslash\overline{D}; \sigma_c, n_c)$ is uniformly elliptic with fixed constants and isotropic as well. In what follows, this issue shall be referred to as the practical feasibility for a cloaking device.\smallskip

\noindent(iii).~For an ideal cloaking device, one can expect the following invisibility performance,
\begin{equation}\label{eq:ii}
u^\infty(\hat x; u^i, (\Omega; \sigma, n))= 0\quad \forall\hat{x}\in\mathbb{S}^{N-1},\ \ \forall u^i.
\end{equation}
However, in practice, especially in order to fulfil the requirements in items (i) and (ii) above, one can relax the ideal cloaking requirement \eqref{eq:ii} to be
\begin{equation}\label{eq:ni}
|u^\infty(\hat x; u^i, (\Omega; \sigma, n))|\ll 1,\quad \forall\hat{x}\in\mathbb{S}^{N-1},\ \ \forall u^i\in\mathscr{W},
\end{equation}
where $\mathscr{W}$ is a set of incident waves consisting of entire solutions to the Helmholtz equation \eqref{eq:Helm1}. That is, near-invisibility can be achieved for scattering measurements made with interrogating waves from the set $\mathscr{W}$. In what follows, this issue shall be referred to as the relaxation and approximation for a cloaking device.\medskip
%\end{itemize}
%\end{enumerate}

In this paper, we aim to explore as much as possible the three issues listed above for a practical cloaking scheme.

\subsection{Existing developments and discussion}

A region of space is said to be \emph{cloaked} if its contents, together with the cloak, are invisible to a particular class of wave measurements. Invisibility cloaking could find striking applications in many areas of science and technology such as radar and sonar, medical imaging, earthquake science and, energy science and engineering, to name just a few. Due to its practical importance, the invisibility cloaking has received great attentions in the literature in recent years, and several cloaking schemes have been proposed, including one based on transformation optics \cite{GLU,GLU2,PenSchSmi,Leo} and another one based on plasmon resonances \cite{AE,MN}.

The plasmonic cloaking uses metamaterials with negative parameters, and can be used to cloak an active source. We refer to \cite{ACKLM,AKL,KLO,KLSW,LLL} and the references therein for the existing developments in this direction. The transformation-optics approach uses the transformation properties of optical parameters via a so-called push-forward to form the cloaking mediums. The transformation-optics mediums for an ideal cloak are nonnegative, but anisotropic and singular, possessing degeneracy and/or blowup singularities. In order to avoid the singular structures, various regularised cloaking schemes have been proposed and investigated in the literature, and instead of ideal cloaking, one considers approximate/near-invisibility cloaking for the regularised constructions. We refer to\cite{Ammari1,Ammari2,Ammari4,BL,BLZ,KOVW,GKLUoe,LiLiuRonUhl,Liu,Liu2,LiuSun} and the references therein for the existing developments in this direction. Though the regularised transformation-optics mediums are nonsingular, they retain the anisotropy, which still poses server difficulties to the practical fabrication. In \cite{GKLLU,GKLU_2,GKLU0}, the authors propose to further approximate the non-singular anisotropic cloaking mediums by isotropic ones using the theory of effective medium and inverse homogenisation. This is closely related to the issues of practical feasibility and relaxation/approximation discussed in Section 1.1. However, the isotropic cloaking medium obtained in \cite{GKLLU,GKLU_2,GKLU0} through the inverse homogenisation are still nearly-singular in the sense that the ellipticity constants of the cloaking material parameters are asymptotic depending on a regularisation parameter.

In this paper, we investigate the non-singular and isotropic cloaking issue through a different perspective. We consider the cloaking device \eqref{eq:cloak} in the ideal case directly for a given cloaking medium. This would lead to an interior transmission eigenvalue problem, which is the one considered in \cite{CCH}. In \cite{CCH}, the authors consider an interior transmission eigenvalue problem for inhomogeneous media containing sound-soft obstacles. We connect the theoretical study in \cite{CCH} with some important practical application on the invisibility cloaking, and propose a generalised interior transmission eigenvalue problem as well. The interior transmission eigenvalue problem arising from the study of inverse scattering theory \cite{ColKre,CM} has also received significant attentions in the literature in recent years \cite{CakCol}. Indeed, the invisibility cloaking problem is naturally connected to a certain interior transmission eigenvalue problem, as shall be discussed in Section 2. Such an observation was also made in a recent paper \cite{BPS} where the authors show that a generic inhomogeneous scatterer with a rectangular support cannot be ideally invisible to every incident wave; that is, it scatters every interrogating wave field. This idea was further picked up in \cite{BCN} where the authors numerically show that ideal invisibility can be achieved for certain wavenumbers with respect to a discrete and finite set of far-field measurement data. In \cite{EH1,EH2}, from a different perspective, the authors show that if the support of a generic inhomogeneous scatterer possesses certain irregularities including a corner, an edge or a circular cone, then it scatters every interrogating wave field. The corner scattering problem has been quantified in \cite{BLU} with stability estimates; see our remarks after Theorem~\ref{thm:3} for more relevant discussions. Our current study shall indicate that an acoustic scatterer, satisfying a certain non-transparency condition, is nearly-invisible with respect to certain incident wave fields. Those incident wave fields are generated from the Herglotz-approximation to certain interior transmission eigenfunctions. Furthermore, based on this study, we propose a novel cloaking scheme using non-singular and isotropic cloaking mediums. The proposed cloaking device takes a three-layer structure with a cloaked region, a lossy layer and a cloaking shell. The target medium in the cloaked region can be arbitrary but regular, whereas the mediums in the lossy layer and the cloaking shell are both non-singular and isotropic. To our best knowledge, the results obtained in the current article are new to the literature. Finally, we refer to \cite{CC,GKLU4,GKLU5,Nor,U2} for comprehensive surveys on the theoretical and experimental progress on invisibility cloaking in the literature; and we also refer to \cite{BPS,CakCol,CCG,CGH,CK,CKP,ColKre,CM,CMS,Kir,PavSyl,S} and the references therein for relevant theoretical and computational studies on the interior transmission eigenvalues.

The rest of the paper is organised as follows. In Section 2, we connect the interior transmission eigenvalue problem with the invisibility cloaking, along with some relevant discussions. In Section 3, we consider the isotropic invisibility cloaking, and establish the near-invisibility results. Section 4 is devoted to numerical verification and demonstration.

\section{Interior transmission eigenvalue problem and Herglotz-approximation}

Let us consider the scattering problem associated with a cloaking device of the form \eqref{eq:cloak}. We first assume that the cloaked region $D$ is isolated from the outer space $\mathbb{R}^N\backslash\overline{D}$. That is, we consider an idealised case that the scattering problem is described by the following Helmholtz system for $u\in H_{loc}^1(\mathbb{R}^N\backslash\overline{D})$,
\begin{equation}\label{eq:Helmi1}
\begin{cases}
& \displaystyle{\nabla\cdot(\sigma_c(x)\nabla u)(x)+\kappa^2 n_c(x) u(x)=0}\qquad\mbox{for}\ \ x\in\Omega\backslash\overline{D},\smallskip\\
& \displaystyle{\Delta u(x)+\kappa^2 u(x)=0}\hspace*{3.2cm}\mbox{for}\ \ x\in\mathbb{R}^N\backslash\overline{\Omega},\smallskip\\
& \mathcal{B} u(x)=0\hspace*{4.7cm}\text{for}\ \ x\in\partial D,\smallskip\\
& \mbox{$u-u^i$ satisfies the Sommerfeld radiation condition},
\end{cases}
\end{equation}
where $\mathcal{B} u:=u$ or $\mathcal{B} u:=\partial u/\partial \nu_{\sigma_c}$ with
\begin{equation}\label{eq:nd}
\frac{\partial u}{\partial\nu_{\sigma_c}}:=\sum_{j,l=1}^N\sigma_c^{jl}\nu_j\partial_{x^l} u\quad\text{on}\ \ \partial D,
\end{equation}
and $\nu=(\nu_j)_{j=1}^N\in\mathbb{S}^{N-1}$ the exterior unit normal vector to $\partial D$. Similar as before, we let $u^s:=u-u^i$ and $u^\infty$ signify the corresponding scattered wave field and the far-field pattern, respectively. If perfect invisibility is achieved for the scattering system \eqref{eq:Helmi1}, namely, $u^\infty(\hat x; (\Omega\backslash\overline{D};\sigma_c, n_c), u^i)\equiv 0$, then by the Rellich theorem (cf. \cite{ColKre}), one has that
\begin{equation}\label{eq:vanish1}
u^s(x)=0\ \ \ \text{for}\ \ x\in\mathbb{R}^N\backslash\overline{\Omega}.
\end{equation}
Next, by using the standard transmission condition across $\partial \Omega$ for the solution $u$ to \eqref{eq:Helmi1}, one has
\begin{equation}\label{eq:vanish2}
u\big|_{\partial \Omega^+}=u\big|_{\partial \Omega^-}\quad\text{and}\quad \frac{\partial u}{\partial\nu}\Big|_{\partial \Omega^+}=\frac{\partial u}{\partial\nu_{\sigma_c}}\Big|_{\partial \Omega^-}:=\sum_{j,l=1}^N\sigma_c^{jl}\nu_j\partial_{x^l} u\Big|_{\partial\Omega^-},
\end{equation}
where $\partial\Omega^{\pm}$ signify the limits taken from outside and inside of $\Omega$, respectively. Applying \eqref{eq:vanish1} to \eqref{eq:vanish2} and by setting $v(x):=u(x)$ and $w(x)=u^i(x)$ for $x\in \Omega$, one can readily show that the following PDE system holds for $(u, v)\in H^1(\Omega\backslash\overline{D})\times H^1(\Omega)$,
\begin{equation}\label{eq:ite}
\begin{cases}
\nabla\cdot(\sigma_c \nabla v)+\kappa^2 n_c v=0\quad & \mbox{in}\ \ \Omega\backslash\overline{D},\smallskip\\
\Delta w+\kappa^2 w=0\quad & \mbox{in}\ \ \Omega,\smallskip\\
\mathcal{B} w=0\quad & \mbox{on}\ \ \partial D,\smallskip\\
\displaystyle{v=w,\qquad \frac{\partial v}{\partial \nu_{\sigma_c}}=\frac{\partial w}{\partial\nu}} \quad & \mbox{on}\ \ \partial\Omega.
\end{cases}
\end{equation}
If for a certain $\kappa\in\mathbb{R}_+$, there exists a nontrivial pair of solutions to \eqref{eq:ite}, then $\kappa$ is called an interior transmission eigenvalue associated with $(\Omega\backslash\overline{D}; \sigma_c, n_c, \mathcal{B})$, and $(v, w)$ is called the corresponding pair of transmission eigenfunctions. It is pointed out that the interior transmission eigenvalue problem \eqref{eq:ite} with $\mathcal{B}w=w$ on $\partial D$ was first proposed and investigated in \cite{CCH} from a theoretical perspective. In this paper, we shall find some important application of \eqref{eq:ite} to invisibility cloaking. 

According to our discussion above, if perfect invisibility is obtained for the scattering system \eqref{eq:Helmi1}, then one has the eigenfunctions for the interior transmission eigenvalue problem \eqref{eq:ite}. However, the existence of interior transmission eigenfunctions for \eqref{eq:ite} does not necessarily imply perfect invisibility for the scattering system \eqref{eq:Helmi1}. Nevertheless, we shall show that near-invisibility can still be achieved under certain cirumstances. To that end, we first need to extend the interior transmission eigenfunction $w$ to the whole space $\mathbb{R}^N$ by the so-called Herglotz-approximation to form an incident wave field for \eqref{eq:Helmi1}. Starting from now and throughout the rest of the paper, we assume that $\Omega$ is of class $C^1$. Define
\begin{equation}\label{eq:app1}
w_\kappa^g(x):=\int_{\mathbb{S}^{N-1}} e^{i\kappa x\cdot \xi} g(\xi)\ ds(\xi),\quad x\in\mathbb{R}^N,\quad g\in L^2(\mathbb{S}^{N-1}).
\end{equation}
$w_\kappa^g$ is called a Herglotz wave function. We have
\begin{thm}[Theorem 2 in \cite{W}]\label{thm:herg1}
Let $\mathbf{W}_\kappa$ denote the space of all Herglotz wave functions of the form \eqref{eq:app1}.
Define, respectively,
\[
W_\kappa(\Omega):=\{u\in C^\infty(\Omega); (-\Delta-\kappa^2) u=0\},
\]
and
\[
\mathbf{W}_\kappa(\Omega):=\{u|_{\Omega}; u\in \mathbf{W}_\kappa\}.
\]
Then $\mathbf{W}_\kappa(\Omega)$ is dense in $W_\kappa(\Omega)\cap H^1(\Omega)$ with respect to the topology induced by the $H^1$-norm.
\end{thm}

We shall also need to introduce the following non-transparency condition for $(\Omega; \sigma_c, n_c, \mathcal{B})$. Consider the following PDE,
\begin{equation}\label{eq:pde1}
\nabla\cdot(\sigma_c\nabla \psi)+\kappa^2 n_c \psi=0\quad\mbox{in}\ \ \Omega\backslash\overline{D},\qquad \mathcal{B}\psi|_{\partial D}=0,\quad \psi\big|_{\partial \Omega}=f\in H^{1/2}(\Omega).
\end{equation}
It is assumed that $\kappa^2$ is not an eigenvalue to \eqref{eq:pde1} in the sense that if $f\equiv 0$, then there exists only a trivial solution to \eqref{eq:pde1}. Hence one has a well-defined DtN map as follows,
\begin{equation}\label{eq:dtn}
\Lambda^\kappa_{\Omega\backslash\overline{D}, \sigma_c, n_c, \mathcal{B}}(f)=\frac{\partial\psi}{\partial\nu_{\sigma_c}}\bigg|_{\partial\Omega}\in H^{-1/2}(\partial\Omega),
\end{equation}
where $\psi\in H^1(\Omega\backslash\overline{D})$ is the unique solution to \eqref{eq:pde1}. Moreover, we assume that the PDE system \eqref{eq:pde1}, with the Dirichlet boundary condition $\psi\big|_{\partial \Omega}=f\in H^{1/2}(\Omega)$ replaced by a Neumann boundary condition $\partial \psi/\partial \nu_{\sigma_c}\big|_{\partial \Omega}=f\in H^{-1/2}(\partial \Omega)$, is also well-posed. Hence, the NtD map $\big(\Lambda^\kappa_{\Omega\backslash\overline{D}, \sigma_c, n_c, \mathcal{B}}\big)^{-1}: H^{-1/2}(\partial \Omega)\rightarrow H^{1/2}(\partial\Omega)$, is also well-defined. Next, we consider the following exterior scattering problem
\begin{equation}\label{eq:pde2}
\begin{cases}
& (\Delta+\kappa^2)u=0\quad\mbox{in}\ \ \mathbb{R}^N\backslash\overline{\Omega},\medskip\\
& u|_{\partial\Omega}=f\in H^{1/2}(\Omega),\\
&\mbox{$u$ satisfies the Sommerfeld radiation condition}.
\end{cases}
\end{equation}
Define the exterior DtN map by
\begin{equation}\label{eq:edtn}
\Gamma^\kappa_{\Omega}(g):=\frac{\partial u}{\partial\nu}\Big|_{\partial\Omega}\in H^{-1/2}(\partial\Omega),
\end{equation}
where $u\in H_{loc}^1(\mathbb{R}^N\backslash\overline{D})$ is the unique solution to \eqref{eq:pde2}. It is known that $\Gamma^\kappa_\Omega$ and its inverse $(\Gamma^\kappa_\Omega)^{-1}$ are both well-defined (cf. \cite{ColKre,Ned}). Then $(\Omega\backslash\overline{D}; \sigma_c, n_c, \mathcal{B})$ is said to satisfy the non-transparency condition associated with $\kappa$ if there holds
\begin{equation}\label{eq:norm}
\|\Lambda^\kappa_{\Omega\backslash\overline{D}, \sigma_c, n_c,\mathcal{B}}\circ(\Gamma^\kappa_\Omega)^{-1}\|_{\mathcal{L}\big({H}^{-1/2}(\partial\Omega), H^{-1/2}(\partial\Omega)\big)}\neq 1.
\end{equation}

\begin{rem}\label{rem:nontransparency}
It can be shown that one must have
\begin{equation}\label{eq:norm11}
\Lambda^\kappa_{\Omega\backslash\overline{D}, \sigma_c, n_c,\mathcal{B}}(f)\neq\Gamma_\Omega^\kappa(f)
\end{equation}
for any $f\in H^{1/2}(\partial\Omega)$ unless $f\equiv 0$. Indeed, for $f\in H^{1/2}(\partial\Omega)$, we let $\psi\in H^{1}(\Omega\backslash\overline{D})$ be the solution to \eqref{eq:pde1}, and $u\in H_{loc}^1(\Omega\backslash\overline{\Omega})$ be the solution to \eqref{eq:pde2}. Set $p=\psi\chi_{\Omega\backslash\overline{D}}+u\chi_{\mathbb{R}^3\backslash\overline{\Omega}}$. If $\Lambda_{\Omega\backslash\overline{D}, \sigma_c, n_c, \mathcal{B}}^\kappa(f)=\Gamma^\kappa_\Omega(f)$, then one readily verifies that $p\in H_{loc}^1(\mathbb{R}^3\backslash\overline{D})$ is the solution to the following system
\begin{equation}\label{eq:pde3}
\begin{cases}
\nabla\cdot (\sigma_c\nabla p)+\kappa^2 n_c p=0\quad & \mbox{in}\ \ \mathbb{R}^N\backslash\overline{D},\medskip\\
 \mathcal{B}p=0\quad & \mbox{on}\ \ \partial D, \\
\mbox{$p$ satisfies the Sommerfeld}&\hspace*{-2mm}\mbox{radiation condition}.
\end{cases}
\end{equation}
Hence, one must have that $p\equiv 0$ in $\mathbb{R}^N\backslash \overline{\Omega}$ (see Section 8, \cite{ColKre}), which readily implies $f\equiv 0$. Hence, it is unobjectionable to claim that the non-transparency condition \eqref{eq:norm} is a generic condition, when $\kappa$ is an interior transmission eigenvalue for $(\Omega\backslash\overline{D}; \sigma_c, n_c, \mathcal{B})$. Nevertheless, we would like to remark that it seems, at least according to our numerical observations, that the non-transparency condition is not satisfied for an inhomogeneous scatterer with an irregular/non-smooth support, say a corner in its support; see our remarks after Theorem~\ref{thm:3} for more relevant discussions.
\end{rem}

Now, we are in a position to present one of the main theorems of this paper which connects the interior transmission eigenvalue problem \eqref{eq:ite} to the invisibility cloaking.

\begin{thm}\label{thm:2}
Let $\kappa_0\in\mathbb{R}_+$ be an interior transmission eigenvalue associated with $(\Omega\backslash\overline{D}; \sigma_c, n_c, \mathcal{B})$, and $(v_0, w_0)\in H^1(\Omega\backslash\overline{D})\times H^1(\Omega)$ be a corresponding pair of eigenfunctions. For any sufficiently small $\varepsilon\in\mathbb{R}_+$, by Theorem~\ref{thm:herg1}, we let $w_{\kappa_0}^g\in\mathbf{W}_{\kappa_0}(\Omega)$ be such that
\begin{equation}\label{eq:herg2}
\|w_{\kappa_0}^g-w_0\|_{H^1(\Omega)}<\varepsilon.
\end{equation}
Consider the scattering problem \eqref{eq:Helmi1} by taking the incident wave field
\begin{equation}\label{eq:scattering2}
u^i=w_{\kappa_0}^g.
\end{equation}
If $(\Omega\backslash\overline{D}; \sigma_c, n_c, \mathcal{B})$ satisfies the non-transparency condition \eqref{eq:norm} with respect to $\kappa_0$, then there holds
\begin{equation}\label{eq:app3}
\Big|u^\infty(\hat x; w^g_{\kappa_0}, (\Omega\backslash\overline{D});\sigma_c, n_c, \mathcal{B}))\Big|\leq C \varepsilon,\quad\forall \hat x\in\mathbb{S}^{N-1},
\end{equation}
where $C$ is a positive constant depending only on $\kappa_0, \sigma_c, n_c$, $\mathcal{B}$ and $\Omega, D$.

\end{thm}

\begin{proof}
Since $\kappa_0\in \mathbb{R}_+$ is an interior transmission eigenvalue associated with $(\Omega\backslash\overline{D}; \sigma_c, n_c, \mathcal{B})$ and $v_0, w_0$ are the corresponding eigenfunctions, we see from \eqref{eq:ite} that
\begin{equation}\label{eq:e1}
\begin{cases}
\nabla\cdot(\sigma_c\nabla v_0)+\kappa_0^2 n_c v_0=0\quad &\mbox{in}\ \ \Omega\backslash\overline{D},\medskip\\
\Delta w_0+\kappa_0^2 w_0=0\quad\ \ \, & \mbox{in}\ \ \Omega,\medskip\\
\mathcal{B}v_0=0\quad & \mbox{on}\ \ \partial D,\medskip\\
v_0=w_0,\quad \displaystyle{\frac{\partial v_0}{\partial\nu_{\sigma_c}}=\frac{\partial w_0}{\partial \nu}}\ \ & \mbox{on}\ \ \partial\Omega,
\end{cases}
\end{equation}
By \eqref{eq:Helmi1} and setting $u^s:=u-u^i=u-w_{\kappa_0}^g$, we clearly have
\begin{equation}\label{eq:e2}
\begin{cases}
\nabla\cdot(\sigma_c\nabla u)+\kappa_0^2 n_c u=0\quad & \mbox{in}\ \ \Omega\backslash\overline{D},\medskip\\
\mathcal{B} u=0\quad & \mbox{on}\ \ \partial D,\medskip\\
\displaystyle{u=w_{\kappa_0}^g+u^s,\quad \frac{\partial u}{\partial\nu_{\sigma_c}}=\frac{\partial w_{\kappa_0}^g}{\partial\nu_{\sigma_c}}+\frac{\partial u^s}{\partial\nu_{\sigma_c}} }\ \ & \mbox{on}\ \ \partial\Omega,
\end{cases}
\end{equation}
and moreover
\begin{equation}\label{eq:e2h}
\begin{cases}
& (\Delta+\kappa_0^2) u^s=0\quad\mbox{in}\ \ \mathbb{R}^N\backslash\overline{\Omega},\medskip\\
& \mbox{$u^s$ satisfies the Sommerfeld radiation condition. }
\end{cases}
\end{equation}
By subtracting \eqref{eq:e1} from \eqref{eq:e2}, and using the transmission conditions on $\partial \Omega$ in \eqref{eq:e1}, we then obtain
\begin{equation}\label{eq:e3}
\begin{cases}
\nabla\cdot(\sigma_c\nabla(u-v_0))+\kappa_0^2 n_c(u-v_0)=0\quad & \mbox{in}\ \ \Omega\backslash\overline{D},\medskip\\
\mathcal{B} (u-v_0)=0\quad & \mbox{on}\ \ \partial D,\medskip\\
u-v_0=u^s+w_{\kappa_0}^g-v_0=u^s+w_{\kappa_0}^g-w_0\quad & \mbox{on}\ \ \partial\Omega,\medskip\\
\displaystyle{\frac{\partial u}{\partial\nu_{\sigma_c}}-\frac{\partial v_0}{\partial\nu_{\sigma_c}}= \frac{\partial u^s}{\partial\nu_{\sigma_c}}+\frac{\partial w_{\kappa_0}^g}{\partial\nu_{\sigma_c}}-\frac{\partial v_0}{\partial\nu_{\sigma_c}}=  \frac{\partial u^s}{\partial\nu_{\sigma_c}}+\frac{\partial w_{\kappa_0}^g}{\partial\nu_{\sigma_c}}-\frac{\partial w_0}{\partial\nu_{\sigma_c}} } \ \ & \mbox{on}\ \ \partial\Omega .
\end{cases}
\end{equation}
Next, by noting the non-transparency condition \eqref{eq:norm}, we first treat the case by assuming that
\begin{equation}\label{eq:norm1}
\|\Lambda_{\Omega\backslash\overline{D}, \sigma_c, n_c,\mathcal{B}}^{\kappa_0}\circ(\Gamma_\Omega^{\kappa_0})^{-1}\|_{\mathcal{L}\big({H}^{-1/2}(\partial\Omega), H^{-1/2}(\partial\Omega)\big)}<1.
\end{equation}
By \eqref{eq:e3}, we have
\begin{equation}\label{eq:d1}
\left(\frac{\partial u}{\partial \nu_{\sigma_c}}-\frac{\partial v_0}{\partial\nu_{\sigma_c}}\right)\bigg|_{\partial \Omega}=\Lambda^{\kappa_0}_{\Omega\backslash\overline{D},\sigma_c, n_c, \mathcal{B}}\big( (u-v_0)|_{\partial\Omega}\big),
\end{equation}
and hence
\begin{equation}\label{eq:d2}
\begin{split}
&\left(\frac{\partial u^s}{\partial\nu_{\sigma_c}}+\frac{\partial w_{\kappa_0}^g}{\partial\nu_{\sigma_c}}-\frac{\partial w_0}{\partial\nu_{\sigma_c}} \right)\bigg|_{\partial\Omega}
=\Lambda^{\kappa_0}_{\Omega\backslash\overline{D}, \sigma_c, n_c, \mathcal{B}}\big((u^s+w_{\kappa_0}^g-w_0)|_{\partial\Omega}\big)\\
=& \Lambda^{\kappa_0}_{\Omega\backslash\overline{D}, \sigma_c, n_c, \mathcal{B}}(u^s|_{\partial\Omega})+\Lambda^{\kappa_0}_{\Omega\backslash\overline{D}, \sigma_c, n_c, \mathcal{B}}\big((w_{\kappa_0}^g-w_0)|_{\partial\Omega}\big)\\
=& \Lambda^{\kappa_0}_{\Omega\backslash\overline{D}, \sigma_c, n_c, \mathcal{B}}\circ\big(\Gamma^{\kappa_0}_{\Omega}\big)^{-1}\left(\frac{\partial u^s}{\partial\nu}\Big|_{\partial\Omega}\right)+\Lambda^{\kappa_0}_{\Omega\backslash\overline{D}, \sigma_c, n_c, \mathcal{B}}\big((w_{\kappa_0}^g-w_0)|_{\partial\Omega}\big),
\end{split}
\end{equation}
which then yields
\begin{equation}\label{eq:d3}
\begin{split}
& \left(I-\Lambda^{\kappa_0}_{\Omega\backslash\overline{D}, \sigma_c, n_c, \mathcal{B}}\circ\big(\Gamma^{\kappa_0}_\Omega\big)^{-1}\right)\left(\frac{\partial u^s}{\partial\nu}\Big|_{\partial\Omega}\right)\\
=&-\left(\frac{\partial w_{\kappa_0}^g}{\partial\nu_{\sigma_c}}-\frac{\partial w_0}{\partial\nu_{\sigma_c}}\right)+\Lambda^{\kappa_0}_{\Omega\backslash\overline{D}, \sigma_c, n_c, \mathcal{B}}\left((w_{\kappa_0}^g-w_0)\big|_{\partial\Omega}\right).
\end{split}
\end{equation}
Combining \eqref{eq:herg2}, \eqref{eq:norm1} and \eqref{eq:d3}, together with straightforward calculations, one readily has
\begin{equation}\label{eq:d4}
\left\|\frac{\partial u^s}{\partial\nu}\right\|_{H^{-1/2}(\partial\Omega)}\leq C\varepsilon,
\end{equation}
where $C$ is a positive constant depending only on $\kappa_0, \sigma_c, n_c$, $\mathcal{B}$ and $\Omega, D$.

For the other case with
\begin{equation}\label{eq:norm2h}
\|\Lambda^{\kappa_0}_{\Omega\backslash\overline{D}, \sigma_c, n_c, \mathcal{B}}\circ\big(\Gamma^{\kappa_0}_\Omega\big)^{-1}\|_{\mathcal{L}\big({H}^{-1/2}(\partial\Omega), H^{-1/2}(\partial\Omega)\big)}>1;
\end{equation}
that is
\begin{equation}\label{eq:norm2}
\|\Gamma^{\kappa_0}_\Omega\circ\big(\Lambda^{\kappa_0}_{\Omega\backslash\overline{D}, \sigma_c, n_c, \mathcal{B}}\big)^{-1}\|_{\mathcal{L}\big({H}^{-1/2}(\partial\Omega), H^{-1/2}(\partial\Omega)\big)}<1,
\end{equation}
by a completely similar argument, one can show that
\begin{equation}\label{eq:d5}
\left\|u^s\right\|_{H^{1/2}(\partial\Omega)}\leq C\varepsilon,
\end{equation}
where $C$ is a positive constant depending only on $\kappa_0, \sigma_c, n_c$, $\mathcal{B}$ and $\Omega, D$.

That is, one either has \eqref{eq:d4} or \eqref{eq:d5} for the exterior scattering system \eqref{eq:e2h}. Finally, by the well-posedness of the scattering problem from a sound-hard or sound-soft obstacle (see \cite{ColKre,Mcl}), one readily has \eqref{eq:app3}.

The proof is complete.

\end{proof}

Now, we let $\mathcal{S}_{\Omega\backslash\overline{D},\sigma_c,n_c,\mathcal{B}}\subset\mathbb{R}_+$ denote the set of all the interior transmission eigenvalues associated with $(\Omega\backslash\overline{D}; \sigma_c, n_c, \mathcal{B})$, and set
\begin{equation}\label{eq:waveset1}
\begin{split}
&\mathcal{W}_{\Omega\backslash\overline{D},\sigma_c,n_c,\mathcal{B}}:=\bigcup_{\kappa\in \mathcal{S}_{\Omega\backslash\overline{D},\sigma_c,n_c,\mathcal{B}}}\Big\{w; (v, w)\in H^1(\Omega\backslash\overline{D})\times H^1(\Omega)\ \text{is a pair of interior}\\
& \text{transmission eigenfunctions corresponding to $\kappa$ associated with $(\Omega\backslash\overline{D}; \sigma_c, n_c, \mathcal{B})$} \Big\}.
\end{split}
\end{equation}
Clearly, $\mathcal{W}_{\Omega\backslash\overline{D},\sigma_c,n_c,\mathcal{B}}$ is a subspace of $W(\Omega):=\cup_{\kappa\in\mathbb{R}_+} W_\kappa(\Omega)$, and by Theorem~\ref{thm:herg1}, for a sufficiently small $\varepsilon\in\mathbb{R}_+$, we let $\mathcal{W}^\varepsilon_{\Omega\backslash\overline{D},\sigma_c,n_c,\mathcal{B}}\subset\mathbf{W}(\Omega):=\cup_{\kappa\in\mathbb{R}_+} \mathbf{W}_\kappa(\Omega)$ be an $\varepsilon$-net of $\mathcal{W}_{\Omega\backslash\overline{D},\sigma_c,n_c,\mathcal{B}}$ in $H^1(\Omega)$ in the sense that for any $w\in \mathcal{W}_{\Omega\backslash\overline{D},\sigma_c,n_c,\mathcal{B}}$, there exists a $w_\kappa^g\in\mathbf{W}(\Omega)$ such that
\begin{equation}\label{eq:herg3}
\|w-w_\kappa^g\|_{H^1(\Omega)}<\varepsilon.
\end{equation}
By Theorem~\ref{thm:2}, one clearly has that
\begin{thm}\label{thm:3}
For any $w_{\kappa_0}^g\in \mathcal{W}^\varepsilon_{\Omega\backslash\overline{D},\sigma_c,n_c,\mathcal{B}}$ associated with an interior transmission eigenvalue $\kappa_0\in \mathcal{S}_{\Omega\backslash\overline{D},\sigma_c,n_c,\mathcal{B}}$, if $(\Omega\backslash\overline{D}; \sigma_c, n_c, \mathcal{B})$ is non-transparent with respect to $\kappa_0$, then there holds
\begin{equation}\label{eq:appg1}
\Big|u^\infty(\hat x; w^g_{\kappa_0}, (\Omega\backslash\overline{D});\sigma_c, n_c, \mathcal{B}))\Big|\leq C \varepsilon,\quad\forall \hat x\in\mathbb{S}^{N-1},
\end{equation}
where $C$ is a positive constant depending only on $\kappa_0, \sigma_c, n_c$, $\mathcal{B}$ and $\Omega, D$.
\end{thm}

Therefore, by Theorem~\ref{thm:3}, as long as the non-transparency condition holds, the cloaked region $D$ with the coating $(\Omega\backslash\overline{D}; \sigma_c, n_c)$ is nearly invisible to the wave interrogation for any incident field from $\mathcal{W}^\varepsilon_{\Omega\backslash\overline{D},\sigma_c,n_c,\mathcal{B}}$, under the assumption that the cloaked region $D$ is isolated from the outer space. Some remarks and discussions are in order.

\begin{enumerate}
\item Clearly, the existence and distribution of interior transmission eigenvalues and eigenfunctions for \eqref{eq:ite} shall be of crucial importance. The existence, discreteness and infiniteness of the interior transmission eigenvalues for \eqref{eq:ite} under general assumptions on $\sigma_c$ and $n_c$ for the case with $\mathcal{B}w=w$ was established in \cite{CCH}. We believe that the case with $\mathcal{B}w=\partial w/\partial \nu_{\sigma_c}$ on $\partial D$ can be treated similarly. Inversely, for certain specific wavenumbers and entire wave fields, the design of $(\Omega\backslash\overline{D}; \sigma_c, n_c, \mathcal{B})$ such that those wavenumbers and wave fields are the respective interior transmission eigenvalues and eigenfunctions shall also be of great interest for customised invisibility cloaking constructions. We leave the theoretical investigation on these issues for our future study. In Section 4, we present extensive numerical experiments to illustrate the discreteness and infiniteness of $\mathcal{S}_{\Omega\backslash\overline{D},\sigma_c,n_c,\mathcal{B}}$, and the validity of Theorems~\ref{thm:2} and \ref{thm:3}.

\item The non-transparency condition \eqref{eq:norm} is critical for the near-invisibility result in Theorem~\ref{thm:3}. Our numerical examples in Section 4 indicate that if $\partial\Omega$ is smooth/regular, then near-invisibility can be achieved at almost all the computed interior transmission eigenvalues; whereas if $\partial \Omega$ is irregular, say possessing a corner, then near-invisibility generically cannot be achieved. The numerical observation is consistent with the theoretical studies in \cite{BLU,BPS,EH1,EH2}. Indeed, as discussed in Section 1.2, it is shown in \cite{BPS,EH1,EH2} that if the support of the scatterer possesses certain irregularities, then it scatters every interrogating wave field. In \cite{BLU}, it is further quantified that the scattered wave field from a corner possesses a positive lower bound. Hence, it might be justifiable to conclude that the non-transparency condition \eqref{eq:norm} holds true for generic acoustic mediums with smooth/regular supports. More quantitatively, according to \eqref{eq:d3},
\begin{equation}\label{eq:lowerbound1}
 \left\|I-\Lambda^{\kappa_0}_{\Omega\backslash\overline{D}, \sigma_c, n_c, \mathcal{B}}\circ\big(\Gamma^{\kappa_0}_\Omega\big)^{-1}\right\|_{\mathcal{L}\big({H}^{-1/2}(\partial\Omega), H^{-1/2}(\partial\Omega)\big)}^{-1}
\end{equation}
should be a regular number when $\kappa_0\in\mathcal{S}_{\Omega\backslash\overline{D}, \sigma_c, n_c, \mathcal{B}}$ and, both $\partial D$ and $\partial \Omega$ are regular/smooth; whereas \eqref{eq:lowerbound1} either blows up or becomes very large when $\kappa_0\in\mathcal{S}_{\Omega\backslash\overline{D}, \sigma_c, n_c, \mathcal{B}}$ and, $\partial D$ or $\partial \Omega$ are irregular. Providing more justifications for the above conclusion is fraught with difficulties, we shall also leave it for further investigation.

\item Theorem~\ref{thm:3} is based on the idealised assumption that the cloaked region $D$ is completely isolated from the outer space. In Section 3, we shall present a finite realisation of $\mathcal{B}$ on $\partial D$ by using properly designed isotropic mediums with loss. The loss mediums are also regular and isotropic, and moreover for the target-independence consideration of the cloaking device, it should enable the object being cloaked to be arbitrary.

\end{enumerate}

\section{Isotropic invisibility cloaking}

Let $(\Omega\backslash\overline{D}; \sigma_c, n_c, \mathcal{B})$ be the one considered in Theorem~\ref{thm:3}. It is assumed that $\partial D$ is of class $C^2$ and, that $n_c>n_0$ is real and $\sigma_c=\mathbf{I}_N$. Let $\Sigma\Subset D$ be a domain of Lipschitz class. Let $\tau\in\mathbb{R}_+$ be an asymptotically small parameter. Set
\begin{equation}\label{eq:lossy1}
\sigma_l(x)=\gamma\tau^{-2} \mathrm{I}_{N\times N},\quad n_l=\alpha+\beta\tau^{-2}\mathrm{i}\quad \mbox{for}\ \ x\in D\backslash\overline{\Sigma},
\end{equation}
if $\mathcal{B}w=w$ in \eqref{eq:ite}; and
\begin{equation}\label{eq:lossy2}
\sigma_l(x)=\gamma\tau^2 \mathrm{I}_{N\times N},\quad n_l=\alpha\tau^2+\beta\tau^2\mathrm{i}\quad \mbox{for}\ \ x\in D\backslash\overline{\Sigma},
\end{equation}
if $\mathcal{B}w=\partial w/\partial\nu$ in \eqref{eq:ite}, where $\alpha, \beta$ and $\gamma$ are positive constants. Consider an acoustic medium configuration as follows,
\begin{equation}\label{eq:cloakr1}
(\mathbb{R}^N; \sigma, n)=(\Sigma;\sigma_a, n_a)\wedge (D\backslash\overline{\Sigma}; \sigma_l, n_l)\wedge (\Omega\backslash\overline{D}; \sigma_c, n_c)\wedge (\mathbb{R}^N\backslash\overline{\Omega}; \mathbf{I}_N, 1),
\end{equation}
where $(\Sigma; \sigma_a, n_a)$ is a regular acoustic medium. Then, we have

\begin{thm}\label{thm:4}
Let $(\mathbb{R}^N; \sigma, n)$ be described in \eqref{eq:cloakr1}. Let $\kappa_0$ and $w_{\kappa_0}^g$ be the same as those in Theorem~\ref{thm:2} satisfying the non-transparency condition \eqref{eq:norm}. Consider the scattering system \eqref{eq:Helm2} corresponding to $(\mathbb{R}^N; \sigma, n)$ with $u^i=w_{\kappa_0}^g$. Then we have
\begin{equation}\label{eq:est1}
\big|u^\infty(\hat x; w_{\kappa_0}^g, (\Omega;\sigma, n))\big|\leq C\big(\tau\|w_{\kappa_0}^g\|_{H^1(\Omega)}+\varepsilon \big),\quad\forall \hat x\in\mathbb{S}^{N-1},
\end{equation}
where $C$ is positive constant depending only on $\Omega, D, \Sigma$, and $\kappa_0, \alpha, \beta, \gamma, n_c$, but independent of $\sigma_a, n_a$ and $\hat x$, $w_{\kappa_0}^g$.
\end{thm}

\begin{proof}
Let us consider the scattering system \eqref{eq:Helm2} corresponding to $(\Omega; \sigma, n)$ described in \eqref{eq:cloakr1}. By multiplying both sides of the equation in \eqref{eq:Helm2} by the complex conjugate of $u$, namely $\overline{u}$, and integrating over $\Omega$, together with the use of integration by parts, we have
\begin{equation}\label{eq:f1}
\begin{split}
& 0= \int_{\Omega}\big[\nabla\cdot(\sigma(x)\nabla u(x))+\kappa_0^2 n(x) u(x) \big]\cdot\overline{u}(x)\ dV(x)\\
=&-\int_{\Omega}(\sigma\nabla u)(x)\cdot(\nabla\overline{u})(x)\ dx+\kappa_0^2\int_{\Omega} n(x) |u(x)|^2\ dV(x)\\
&\qquad\qquad\qquad+\int_{\partial\Omega}\frac{\partial u}{\partial\nu}(x)\cdot\overline{u}(x)\ ds(x)
\end{split}
\end{equation}
By taking the imaginary parts of both sides of \eqref{eq:f1}, one can easily verify that
\begin{equation}\label{eq:f2}
\left|\kappa_0^2\Im n_l\int_{D\backslash\overline{\Sigma}}|u(x)|^2 dV(x)\right|\, dx\leq \left|\int_{\partial\Omega}\frac{\partial u}{\partial \nu}(x)\cdot\overline{u}(x)\, ds(x)  \right|.
\end{equation}
Hence, if $n_l$ is given in \eqref{eq:lossy1}, one has
\begin{equation}\label{eq:f3}
\|u\|_{L^2(D\backslash\overline{\Sigma})}\leq \kappa_0\sqrt{\beta}\tau\left|\int_{\partial\Omega}\frac{\partial u}{\partial \nu}\cdot \overline{u}\, ds\right|^{1/2},
\end{equation}
whereas if $n_l$ is given in \eqref{eq:lossy2}, then one has
\begin{equation}\label{eq:f4}
\|u\|_{L^2(D\backslash\overline{\Sigma})}\leq \kappa_0\sqrt{\beta}\tau^{-1}\left|\int_{\partial\Omega}\frac{\partial u}{\partial \nu}\cdot \overline{u}\, ds\right|^{1/2}.
\end{equation}

Next, we first consider the case with $(D\backslash\overline{\Sigma}; \sigma_l, n_l)$ given in \eqref{eq:lossy1}; that is, $\mathcal{B} v=v$ in \eqref{eq:ite}. We claim that

\medskip
\noindent\underline{\bf Claim 1.}~Let $u^+$ and $u^-$ denote the traces of $u$ on $\partial D$ when approaching $\partial D$, respectively, from the exterior and interior of $D$. Then there holds
\begin{equation}\label{eq:f5}
\|u^+\|_{H^{-1/2}(\partial D)}\leq C \tau \left|\int_{\partial\Omega}\frac{\partial u}{\partial \nu}\cdot \overline{u}\, ds\right|^{1/2},
\end{equation}
where $C$ is a positive constant depending only on $\kappa_0$, $\alpha, \beta, \gamma,$ and $D, \Sigma$, but independent of $(\Sigma; \sigma_a, n_a)$.
\medskip

\noindent\underline{\bf Proof of Claim 1.}~~By the transmission conditions across $\partial D$ for $u\in H_{loc}^1(\mathbb{R}^N)$, one clearly has $u^+=u^-$, and hence in order to prove \eqref{eq:f5}, it suffices for us to verify that the same estimate holds for $u^-$. To that end, we make use of the following duality identity,
\begin{equation}\label{eq:f6}
\|u^-\|_{H^{-1/2}(\partial D)}=\sup_{\|\varphi\|_{H^{1/2}(\partial D)}\leq 1}\left|\int_{\partial D} u^-\cdot \varphi\ ds \right|.
\end{equation}
For any $\varphi\in H^{1/2}(\partial D)$, there exists $w\in H^2(D)$ such that (cf. \cite{Wlo})\smallskip
\begin{enumerate}
\item[i).]~~$\displaystyle{\frac{\partial w}{\partial\nu}=\varphi}$ and $w=0$ on $\partial D$;\smallskip

\item[ii).]~~$w=0$ on $\Sigma$, and $\|w\|_{H^2(D)}\leq C\|\varphi\|_{H^{1/2}(\partial D)}$;
\end{enumerate}
where $C$ is a positive constant depending only on $\Sigma$ and $D$. Then we have
\begin{equation}\label{eq:f7}
\begin{split}
& \int_{\partial D} u^-\cdot \varphi\ ds=\int_{\partial D} u^-\cdot\frac{\partial w}{\partial \nu}-\frac{\partial u^-}{\partial\nu}\cdot w\ ds\\
=& \int_{D\backslash\overline{\Sigma}} u(x)\cdot\Delta w(x)-\Delta u(x)\cdot w(x)\ dV(x).
\end{split}
\end{equation}
Noting that in $D\backslash\overline{\Sigma}$, one has
\begin{equation}\label{eq:f8}
\nabla\cdot(\sigma_l\nabla u)(x)+\kappa_0^2 n_l u(x)=0\quad\mbox{for}\ \ x\in D\backslash\overline{\Sigma},
\end{equation}
which together with \eqref{eq:lossy1} readily implies that
\begin{equation}\label{eq:f9}
\Delta u=-\kappa_0^2(\alpha\tau^2+\mathrm{i}\beta)u\quad\mbox{in}\ \ D\backslash\overline{\Sigma}.
\end{equation}
By plugging \eqref{eq:f9} into \eqref{eq:f7}, together with the use of Schwarz inequality, we then have
\begin{equation}\label{eq:f10}
\begin{split}
\left|\int_{\partial D}u^-\cdot \varphi\ ds\right|\leq & \left(1+\kappa_0^2|\alpha+\mathrm{i}\beta| \right)\|u\|_{L^2(D\backslash\overline{\Sigma})}\|w\|_{H^2(D)}\\
\leq & C \left(1+\kappa_0^2|\alpha+\mathrm{i}\beta| \right)\|u\|_{L^2(D\backslash\overline{\Sigma})}\|\varphi\|_{H^{1/2}(\partial D)},
\end{split}
\end{equation}
which in combination with \eqref{eq:f3}, along with straightforward calculations, readily verifies \eqref{eq:f5} for $u^-$, and hence for $u^+$ as well. This completes the proof of Claim 1. \hfill $\Box$

\medskip

For the other case with $(D\backslash\overline{\Sigma}; \sigma_l, n_l)$ given in \eqref{eq:lossy2}; that is, $\mathcal{B} w=\partial w/\partial\nu$ in \eqref{eq:ite}. We can show that

\medskip
\noindent\underline{\bf Claim 2.}~Let $\partial u^+/\partial\nu$ and $\partial u^-/\partial\nu$ denote the traces of $\partial u/\partial\nu$ on $\partial D$ when approaching $\partial D$, respectively, from the exterior and interior of $D$. Then there holds
\begin{equation}\label{eq:f11}
\left\|\frac{\partial u^+}{\partial\nu}\right\|_{H^{-3/2}(\partial D)}\leq C \tau \left|\int_{\partial\Omega}\frac{\partial u}{\partial \nu}\cdot \overline{u}\, ds\right|^{1/2},
\end{equation}
where $C$ is a positive constant depending only on $\kappa_0$, $\alpha, \beta, \gamma$ and $D, \Sigma$, but independent of $(\Sigma; \sigma_a, n_a)$.
\medskip

\noindent\underline{\bf Proof of Claim 2.}~~Claim 2 can be proved by following a similar argument to that for the proof of Claim 1. Indeed, we first estimate $\|\partial u^-/\partial\nu\|_{H^{-3/2}(\partial D)}$, and make use of the following duality identity
\begin{equation}\label{eq:f6h}
\left\|\frac{\partial u^-}{\partial\nu}\right\|_{H^{-3/2}(\partial D)}=\sup_{\|\varphi\|_{H^{3/2}(\partial D)}\leq 1}\left|\int_{\partial D} \frac{\partial u^-}{\partial\nu}\cdot \phi\ ds \right|.
\end{equation}
For any $\phi\in H^{3/2}(\partial D)$, there exists $w\in H^2(D)$ such that (cf. \cite{Wlo})\smallskip
\begin{enumerate}
\item[i).]~~$\displaystyle{\frac{\partial w}{\partial\nu}=0}$ and $w=\phi$ on $\partial D$;\smallskip

\item[ii).]~~$w=0$ on $\Sigma$, and $\|w\|_{H^2(D)}\leq C\|\phi\|_{H^{3/2}(\partial D)}$;
\end{enumerate}
where $C$ is a positive constant depending only on $\Sigma$ and $D$. Then by a similar argument to that for the proof of Claim 1 in deriving \eqref{eq:f10}, together with the use of \eqref{eq:lossy2}, one can show that
\begin{equation}\label{eq:f12}
\left\|\frac{\partial u^-}{\partial \nu} \right\|_{H^{-3/2}(\partial D)}\leq \frac{C}{\sqrt{\beta}}(\gamma+\kappa_0^2|\alpha+\mathrm{i}\beta|)\tau^{-1}\left| \int_{\partial\Omega}\frac{\partial u}{\partial\nu}\cdot\overline{u}\, ds\right|^{1/2}.
\end{equation}
Finally, by using the following transmission condition across $\partial D$
\[
\sigma_l\frac{\partial u^-}{\partial \nu}=\frac{\partial u^+}{\partial\nu}\quad\mbox{on}\ \ \partial D,
\]
and \eqref{eq:f12}, one readily sees \eqref{eq:f11}. This completes the proof of Claim 2. \hfill $\Box$

\medskip

We are ready to prove \eqref{eq:est1}. Let us first treat the case with $\mathcal{B} w=w$ in \eqref{eq:ite}. Set
\begin{equation}\label{eq:g1}
f= u^+|_{\partial D}\in H^{1/2}(\partial D)\ \mbox{and}\ u^s(x)=u(x)-w_{\kappa_0}^g(x)\ \ \mbox{for}\ x\in\mathbb{R}^N\backslash\overline{D},
\end{equation}
It is easily seen that
\begin{equation}\label{eq:g2}
\begin{cases}
(\Delta+\kappa_0^2 n) u=0\quad & \mbox{in}\ \ \mathbb{R}^N\backslash\overline{D},\smallskip\\
u=f\qquad & \mbox{on}\ \ \partial D,\smallskip\\
u=u^s+w_{\kappa_0}^g\qquad & \mbox{in}\ \ \mathbb{R}^N\backslash\overline{D},\smallskip\\
\mbox{$u^s$ satisfies the Sommerfeld}&\hspace*{-2.5mm}\mbox{radiation condition}.
\end{cases}
\end{equation}
We also introduce $u_c\in H_{loc}^1(\mathbb{R}^N\backslash\overline{D})$ and $u_c^s(x):=u_c(x)-w_{\kappa_0}^g(x)$, $x\in\mathbb{R}^N\backslash\overline{D}$, satisfying
\begin{equation}\label{eq:g3}
\begin{cases}
(\Delta+\kappa_0^2 n) u_c=0\quad & \mbox{in}\quad\mathbb{R}^N\backslash\overline{D},\medskip\\
u_c=0\quad & \mbox{on}\ \ \partial D,\\
\mbox{$u_c^s$ satisfies the Sommerfeld}&\hspace*{-2.5mm}\mbox{radiation condition}.
\end{cases}
\end{equation}
Set
\begin{equation}\label{eq:g4}
\widetilde{u}(x)=u(x)-u_c(x),\quad x\in\mathbb{R}^N\backslash\overline{D}.
\end{equation}
Then $\widetilde{u}\in H_{loc}^1(\mathbb{R}^N\backslash\overline{D})$ satisfies
\begin{equation}\label{eq:g5}
\begin{cases}
&(\Delta+\kappa_0^2 n_c)\widetilde{u}=0\quad\mbox{in}\ \ \mathbb{R}^N\backslash\overline{D},\smallskip\\
& \widetilde{u}=f\quad\mbox{on}\ \ \partial D,\smallskip\\
&\mbox{$\widetilde u$ satisfies the Sommerfeld radiation condition. }
\end{cases}
\end{equation}
By Lemma~\ref{lem:aux} in the following, we have
\begin{equation}\label{eq:g6}
\|\widetilde u\|_{H^{1/2}(\partial \Omega)}\leq C\|f\|_{H^{-1/2}(\partial D)},
\end{equation}
where $C$ is a positive constant depending only on $\Omega, D$ and $\kappa_0, n_c$. That is,
\begin{equation}\label{eq:g7}
\|u^s-u_c^s\|_{H^{1/2}(\partial\Omega)}=\|u-u_c\|_{H^{1/2}(\partial\Omega)}=\|\widetilde{u}\|_{H^{1/2}(\partial\Omega)}\leq C\|f\|_{H^{-1/2}(\partial D)},
\end{equation}
which in turn implies that
\begin{equation}\label{eq:g8}
\|u^s\|_{H^{1/2}(\partial \Omega)}\leq \|u_c^s\|_{H^{1/2}(\partial\Omega)}+C\|f\|_{H^{-1/2}(\partial\Omega)}.
\end{equation}
Next, by the argument in the proof of Theorem~\ref{thm:2} (cf. \eqref{eq:d4} and \eqref{eq:d5}), we see that
\begin{equation}\label{eq:x1}
\|u_c^s\|_{H^{1/2}(\partial \Omega)}\leq C\varepsilon.
\end{equation}
By applying \eqref{eq:x1}, \eqref{eq:f5} to \eqref{eq:g8}, we have
\begin{equation}\label{eq:g9}
\|u^s\|_{H^{1/2}(\partial\Omega)}\leq C\varepsilon+C \tau\left|\int_{\partial\Omega}\frac{\partial u}{\partial\nu}\cdot\overline{u}\ ds \right|^{1/2}.
\end{equation}
Next, by Green's formula and Schwarz inequality, one can deduce as follows
\begin{equation}\label{eq:g10}
\begin{split}
&\left|\int_{\partial\Omega}\frac{\partial u}{\partial\nu}\cdot\overline{u}\ ds \right|\\
=&\left| \int_{\Omega}\left(\frac{\partial u^s}{\partial\nu}+\frac{\partial w_{\kappa_0}^g}{\partial\nu} \right)\cdot(\overline{u^s}+\overline{w_{\kappa_0}^g})\ ds  \right|\\
\leq & \left|\int_{\partial\Omega}\frac{\partial u^s}{\partial\nu}\cdot\overline{u^s}\ ds\right|+\left|\int_{\partial\Omega}\frac{\partial u^s}{\partial\nu}\cdot\overline{w_{\kappa_0}^g}\ ds\right|+\left|\int_{\partial\Omega}\frac{\partial w_{\kappa_0}^g}{\partial\nu}\cdot\overline{u^s}\ ds\right|\\
&\quad +\big|\|\nabla w_{\kappa_0}^g\|^2_{L^2(\Omega)}-\kappa_0^2\|w_{\kappa_0}^g\|^2_{L^2(\Omega)} \big|\\
\leq & \left\|\frac{\partial u^s}{\partial\nu}\right\|_{H^{-1/2}(\partial\Omega)}\|u^s\|_{H^{1/2}(\partial\Omega)}+\left\|\frac{\partial u^s}{\partial \nu} \right\|_{H^{1/2}(\partial\Omega)}\|w_{\kappa_0}^g\|_{H^{1/2}(\partial\Omega)}\\
&\left\|\frac{\partial w_{\kappa_0}^g}{\partial\nu} \right\|_{H^{-1/2}(\partial\Omega)}\|u^s\|_{H^{1/2}(\partial\Omega)}+\|\nabla w_{\kappa_0}^g\|^2_{L^2(\Omega)}+\kappa_0^2\|w_{\kappa_0}^g\|^2_{L^2(\Omega)} \\
\leq & \|\Gamma_\Omega^{\kappa_0}\|_{\mathcal{L}(H^{1/2}(\partial\Omega), H^{-1/2}(\partial\Omega))}\|u^s\|^2_{H^{1/2}(\partial\Omega)}\\
&+\left\|\frac{\partial u^s}{\partial\nu} \right\|_{H^{-1/2}(\partial\Omega)}^2+\frac{1}{4}\|w_{\kappa_0}^g\|^2_{H^{1/2}(\partial\Omega)}+\|u^s\|_{H^{1/2}(\partial\Omega)}^2\\
&+\frac{1}{4}\left\|\frac{\partial w_{\kappa_0}^g}{\partial\nu}\right\|_{H^{-1/2}(\partial\Omega)}^2+\|\nabla w_{\kappa_0}^g\|^2_{L^2(\Omega)}+\kappa_0^2\|w_{\kappa_0}^g\|^2_{L^2(\Omega)}\\
&\leq \|\Gamma^{\kappa_0}_\Omega\|_{\mathcal{L}(H^{1/2}(\partial\Omega), H^{-1/2}(\partial\Omega))}\|u^s\|^2_{H^{1/2}(\partial\Omega)}
+\|u^s\|_{H^{1/2}(\partial\Omega)}^2\\
&+\|\Gamma^{\kappa_0}_\Omega\|^2_{\mathcal{L}(H^{1/2}(\partial\Omega), H^{-1/2}(\partial\Omega))}\|u^s\|^2_{H^{1/2}(\partial\Omega)}+C\|w_{\kappa_0}^g\|_{H^1(\Omega)}^2
\end{split}
\end{equation}
where $C$ is positive constant depending only on $\kappa_0$ and $\Omega$. By plugging \eqref{eq:g10} into \eqref{eq:g9}, we then have
\begin{equation}\label{eq:g11}
\begin{split}
\|u^s\|_{H^{1/2}(\partial\Omega)}\leq & C\varepsilon+C\tau\Big(1+\|\Gamma_\Omega\|_{\mathcal{L}(H^{1/2}(\partial\Omega), H^{-1/2}(\partial\Omega) )   }  \Big) \|u^s\|_{H^{1/2}(\partial\Omega)}\\
&+C\tau\|w_{\kappa_0}^g\|_{H^1(\Omega)}.
\end{split}
\end{equation}
For sufficiently small $\tau$, we obviously have from \eqref{eq:g11} that
\begin{equation}\label{eq:g12}
\|u^s\|_{H^{1/2}(\partial\Omega)}\leq C\big(\varepsilon+\tau\|w_0^g\|_{H^1(\Omega)}\Big)
\end{equation}
Finally, by the well-posedness of the acoustic scattering problem from a sound-soft obstacle (see, \cite{ColKre,Mcl}), one readily has \eqref{eq:est1} from \eqref{eq:g12}.

The other case with $\mathcal{B} v=\partial v/\partial v$ in \eqref{eq:ite} can be proved by following a similar argument. In what follows, we only sketch the necessary modifications that will be needed. In \eqref{eq:g1}, $f=u^+|_{\partial D}\in H^{1/2}(\partial D)$ should be modified to be $f=\partial u^+/\partial\nu\in H^{-1/2}(\partial D)$; and in \eqref{eq:g2} and \eqref{eq:g5}, the Dirichlet boundary conditions on $\partial D$ should be modified to be the corresponding Neumann ones. Then, by following similar arguments in \eqref{eq:g6}--\eqref{eq:g10}, along with the use of \eqref{eq:f11}, one can arrive at a similar estimate to \eqref{eq:g11} for $\|u^s\|_{H^{1/2}(\partial\Omega)}$, and hence verify \eqref{eq:est1}.

The proof is complete.
\end{proof}

The following lemma was needed in the proof of Theorem~\ref{thm:4}.

\begin{lem}\label{lem:aux}
Let $(\mathbb{R}^N\backslash\overline{D}; \sigma_c, n_c)$ be described in \eqref{eq:cloakr1}. Let $u\in H_{loc}^1(\mathbb{R}^N\backslash\overline{D})$ be the (unique) solution to
\begin{equation}\label{eq:gg5}
\begin{cases}
&(\Delta+\kappa_0^2 n){u}=0\qquad\ \mbox{in}\ \ \mathbb{R}^N\backslash\overline{D},\smallskip\\
& {u}=f\in H^{1/2}(\partial D)\quad\mbox{on}\ \ \partial D,\smallskip\\
&\mbox{$ u$ satisfies the Sommerfeld radiation condition. }
\end{cases}
\end{equation}
Then there holds
\begin{equation}\label{eq:gg6}
\|u\|_{H^1(\Omega\backslash\overline{D})}\leq C\|f\|_{H^{-1/2}(\partial D)},
\end{equation}
where $C$ is positive constant depending only on $\kappa_0, n_c$ and $\Omega, D$. If the Dirichlet boundary condition in \eqref{eq:gg5} is replaced by
\[
\frac{\partial u}{\partial\nu}=f\in H^{-1/2}(\partial D)\quad\mbox{on}\ \ \partial D,
\]
then there holds
\begin{equation}\label{eq:gg7}
\|u\|_{H^1(\Omega\backslash\overline{D})}\leq C\|f\|_{H^{-3/2}(\partial D)},
\end{equation}
where $C$ is positive constant depending only on $\kappa_0, n_c$ and $\Omega, D$.
\end{lem}
\begin{proof}
The proof follows from a completely similar argument to that for the proof of Theorem 2.2 in \cite{KP} by making use of the integral equation method, along with straightforward and necessary modifications by using the mapping properties of the involved potential operators in the Sobolev spaces in \cite{Mcl,Ned}.
\end{proof}

\section{Numerical experiments}

In this section, we conduct extensive numerical experiments to illustrate and verify the theoretical results in Sections 2 and 3. Specifically, the following three goals shall be achieved in our numerical study. 
\begin{itemize}
\item[i).] The existence and discreteness of interior transmission eigenvalues of \eqref{eq:ite} shall be numerically verified.
\item[ii).] The results in Theorems~\ref{thm:2} and \ref{thm:3} shall be numerically demonstrated and verified. 
\item[iii).] The results in Theorem~\ref{thm:4} shall be numerically demonstrated and verified. 
\end{itemize}
Different shapes of the scatterer such as circle, ellipse and square are considered. Throughout, we take $\sigma_c=1$ without loss of generality.

\subsection{Computation of the interior transmission eigenvalues}

Following our earlier discussion, we consider the following system of PDEs associated with the interior transmission eigenvalue problem \eqref{eq:ite},
\begin{equation}\label{eq:iten1}
\begin{cases}
\Delta v+\kappa^2 n_c v=0\quad & \mbox{in}\ \ \Omega\backslash\overline{D},\smallskip\\
\Delta w+\kappa^2 w=0\quad & \mbox{in}\ \ \Omega,\smallskip\\
v=0\quad & \mbox{on}\ \ \partial D,\smallskip\\
\displaystyle{v=w,\quad\frac{\partial v}{\partial \nu}=\frac{\partial w}{\partial\nu}} \quad & \mbox{on}\ \ \partial\Omega.
\end{cases}
\end{equation}
Here, we impose the homogeneous Dirichlet boundary condition on $\partial D$ for the subsequent discussion on the finite element discretisation of \eqref{eq:iten1},  and the homogeneous Neumann boundary condition can be treated similarly.
The weak formulation for (\ref{eq:iten1}) is read as follows for any $\phi\in H^1(\Omega)$,
\begin{eqnarray}\label{weak1}
&\int_{\partial\Omega \cup \partial D } \frac{\partial v}{\partial \nu}\phi\, ds-\int_{\Omega\backslash D}\nabla v\cdot \nabla \phi\, dV=-\int_{\Omega\backslash D} k^2n_c v\phi\, dV,\smallskip\\
& \int_{\partial\Omega}\frac{\partial w}{\partial \nu}\phi\, ds-\int_\Omega \nabla w\cdot\nabla \phi\, dV=-\int_\Omega k^2 w\phi\, dV. \label{weak2}
\end{eqnarray}
Subtracting (\ref{weak2}) from (\ref{weak1}) and using the boundary conditions in (\ref{eq:iten1}), we have
\begin{equation}\label{weak3}
\int_{\Omega\backslash D} \nabla v\cdot \nabla \phi\, dV-\int_{\Omega}\nabla w\cdot \nabla \phi\, dV-\int_{\partial D} \frac{\partial v}{\partial \nu}\phi\, ds=\int_{\Omega\backslash D} \kappa^2n_cv\phi\, dV-\int_{\Omega} \kappa^2 w \phi\, dV.
\end{equation}
For the numerical discretisation of \eqref{weak3}, we make use of the standard Lagrange finite elements. Define
\begin{eqnarray}
S_h&=&\mbox{the space of continuous piecewise $p$-degree finite elements on $\Omega$},\\
S_h^0&=&S_h\cap H_0^1(\Omega)\nonumber\\
&=&\mbox{the subspace of functions in $S_h$ with vanishing DoF on $\partial \Omega$},\\
S_h^B&=&\mbox{the subspace of functions in $S_h$ with vanishing DoF in $\Omega$},\\
S_h^D&=& S_h^0 \cap H_0^1(\Omega\backslash D).
\end{eqnarray}
where DoF stands for degrees of freedom. We explicitly enforce the Dirichlet boundary condition in (\ref{eq:ite}) by letting
\begin{eqnarray}
v_h&=&v_{0,h}+v_{B,h}, \quad \mbox{where } v_{0,h}\in S_h^D \mbox{ and } v_{B,h}\in S_h^B,\\
w_h&=&w_{0,h}+v_{B,h},  \quad \mbox{where } w_{0,h} \in S_h^0.
\end{eqnarray}
In (\ref{weak1}), by letting the test function $\xi_h\in S_h^D$, we obtain the standard weak formulation for $v_h$ as
\begin{equation}\label{dis1}
\int_{\Omega\backslash D}\nabla(v_{0,h}+v_{B,h})\nabla \xi_h dx=\int_{\Omega\backslash D} k^2n_c (v_{0,h}+v_{B,h})\xi_h\, dV
\end{equation}
for all $\xi_h\in S_h^D$.
Analogously, by letting the test function $\eta_h\in S_h^0$, we obtain the weak formulation for $w_h$ as
\begin{equation}\label{dis2}
\int_\Omega \nabla (w_{0,h}+v_{B,h})\nabla \eta_h\, dV=\int_\Omega k^2 (w_{0,h}+v_{B,h})\eta_h\, dV
\end{equation}
for all $\eta_h\in S_h^0$.
For (\ref{weak3}), by letting $\phi_h\in S_h^B$, we have
\begin{eqnarray}\label{dis3}
\int_{\Omega\backslash D}\nabla (v_{0,h}+v_{B,h})\nabla \phi_h\, dV-\int_{\Omega}\nabla (w_{0,h}+v_{B,h})\nabla \phi_h\, dV\nonumber\\
=\int_{\Omega\backslash D}k^2n_c(v_{0,h}+v_{B,h})\phi_h\, dV-\int_{\Omega}k^2(w_{0,h}+v_{B,h}) \phi_h\, dV.
\end{eqnarray}
Let $N_h,N_h^0,N_h^B$ and $N_h^D$ be the dimensions of $S_h$, $S_h^0$, $S_h^B$ and $S_h^D$, respectively. In addition, we choose $\{\xi_1,\cdots,\xi_{N_h}\}$ to be the finite element basis for $S_h$. We define the following matrices
\begin{center}
\begin{tabular}{l|l}
\hline
$S$ &  stiffness matrix, $(S)_{j,\ell}=\int_\Omega\nabla\xi_\ell\cdot \nabla\xi_{j}\, dV$\\
$S_D$ &  stiffness matrix, $(S)_{j,\ell}=\int_{\Omega\backslash D}\nabla\xi_\ell\cdot \nabla\xi_{j}\, dV$\\
$M_n$& mass matrices, $(M_n)_{j,\ell}=\int_\Omega n_c \xi_\ell \xi_{j}\, dV$\\
$M$& mass matrices, $(M)_{j,\ell}=\int_\Omega \xi_\ell\xi_{j}\, dV$\\
$M_{nD}$& mass matrices, $(M_{nD})_{j,\ell}=\int_{\Omega\backslash D}n_c \xi_\ell \xi_{j}\, dV$\\
\hline
\end{tabular}
\end{center}
Combining \eqref{dis1},\eqref{dis2} and \eqref{dis3}, the discrete counterpart associated with \eqref{eq:iten1} is to solve the following generalised eigenvalue problem
\begin{equation}
{\mathcal A}\vec{\mathbf{x}}=\kappa^2{\mathcal B}\vec{\mathbf{x}},
\end{equation}
where the matrices ${\mathcal A}$ and ${\mathcal B}$ are given block-wisely by
\[
{\mathcal A}=\left(\begin{array}{ccc}
S_D^{N_h^D\times N_h^D}&0&S_D^{N_h^D \times N_h^B} \\
0&S^{N_h^0\times N_h^0}&S^{N_h^0 \times N_h^B}\\
S_D^{N_h^B \times N_h^D}&-S^{N_h^B \times N_h^0}&S_D^{N_h^B \times N_h^B}-S^{N_h^B \times N_h^B}\end{array}\right),
\]
and
\[
{\mathcal B}=\left(\begin{array}{ccc}
M_{nD}^{N_h^D\times N_h^D}&0&M_{nD}^{N_h^D \times N_h^B} \\
0&M^{N_h^0\times N_h^0}&M^{N_h^0 \times N_h^B}\\
M_{nD}^{N_h^B \times N_h^D}&-M^{N_h^B \times N_h^0}&M_{nD}^{N_h^B \times N_h^B}-M^{N_h^B \times N_h^B}\end{array}\right).
\]

In all the numerical examples, we set $n_c=16$. Table \ref{table1} and Table \ref{table2} present the interior transmission eigenvalues for circles  ($\{\Omega: \|x\|<1\}$ and $\overline{\{D: \|x\|<0.5\}}$) with the homogeneous Dirichlet boundary condition on $\partial D$, and the results converge when we decrease the size of the mesh. In Table \ref{table2}, complex eigenvalues exist due to the non-selfadjointness of the interior eigenvalue problem \eqref{eq:iten1}. Table \ref{table3} gives the convergence test for circles with the homogeneous Neumann boundary condition imposed on $\partial D$.

Table \ref{table4} presents the interior transmission eigenvalues for ellipses ($\{\Omega: (x^1)^2+(\frac{x^2}{1.2})^2<1 \}$ and $\overline{\{D:(\frac{x^1}{0.5})^2+(\frac{x^2}{0.6})^2<1\}}$) with $h=0.1$, and the first line is the result for the Dirichlet boundary condition, while the second line
for the Neumann boundary condition.

Table \ref{table5} presents the interior transmission eigenvalues for squares ($\{\Omega: |x^1|<1\cap |x^2|<1 \}$ and $\overline{\{D: |x^1|<0.5\cap |x^2|<0.5\}}$) with $h=0.1$, and the first line is the result for the Dirichlet boundary condition, while the second line for the Neumann boundary condition.

 \begin{table}
 \begin{center}
\begin{tabular}{l|l|l|l|l|l}
\hline
$h=0.2$ & 0.352664 &0.353659 &0.519072& 0.518733 & 0.743794  \\
$h=0.1$ & 0.353965 &0.354349& 0.517122& 0.517444& 0.738215  \\
$h=0.05$& 0.353811 &0.353770& 0.516314& 0.516468&0.736280\\
\hline
\end{tabular}
\caption{Five interior transmission eigenvalues closest to 1 in the circle geometry with $n_c=16$ and the Dirichlet boundary condition. }
\label{table1}
\end{center}
 \end{table}

  \begin{table}
 \begin{center}
\begin{tabular}{l|l|l }
\hline
$h=0.2$ & 2.457778-0.466055i &  2.457778+0.466055i\\
$h=0.1$ & 2.402496-0.415131i &  2.402496+0.415131i \\
$h=0.05$& 2.387120-0.401892i &  2.387120+0.401892i\\
\hline
\end{tabular}
\caption{Complex interior transmission eigenvalues closest to 2.5 in the circle geometry with $n_c=16$ and the Dirichlet boundary condition.}
\label{table2}
\end{center}
 \end{table}

 \begin{table}
 \begin{center}
\begin{tabular}{l|l|l|l|l|l}
\hline
$h=0.2$ & 1.748573 &1.757866 &1.802944& 1.811505 & 1.991750  \\
$h=0.1$ & 1.671757 &1.675173& 1.721631& 1.723695& 1.890939  \\
$h=0.05$& 1.646361 &1.647434& 1.692928& 1.694515&1.842568\\
\hline
\end{tabular}
\caption{Five smallest interior transmission eigenvalues in the circle geometry with $n_c=16$ and the Neumann boundary condition. }
\label{table3}
\end{center}
 \end{table}

 \begin{table}
 \begin{center}
\begin{tabular}{l|l|l|l|l|l|l}
\hline
Dirichlet & 2.097681 &2.165191 &2.207713& 2.220829 & 2.225652-0.359758i& 2.225652+0.359758i \\
Neumann  & 1.483283 &1.533343& 1.594063& 1.597701 &1.747153& 1.751044  \\
\hline
\end{tabular}
\caption{Interior transmission eigenvalues in the ellipse geometry with $n_c=16$ and $h=0.1$: the first line lists the six eigenvalues closet to 2 with the Dirichlet boundary condition; the second line lists the smallest six eigenvalues for the Neumann boundary condition.}
\label{table4}
\end{center}
 \end{table}

  \begin{table}
 \begin{center}
\begin{tabular}{l|l|l|l|l|l|l}
\hline
Dirichlet & 1.800246-0.198428i &1.800246+0.198428i  &2.170438& 2.317611& 2.431338& 2.471902 \\
Neumann  & 0.761138 &1.192171& 1.649127& 1.678710 &1.679462& 1.736597  \\
\hline
\end{tabular}
\caption{Interior transmission eigenvalues in the square geometry with $n_c=16$ and $h=0.1$: the first line lists the six eigenvalues closet to 2 with the Dirichlet boundary condition; the second line lists the smallest six eigenvalues for the Neumann boundary condition.}
\label{table5}
\end{center}
 \end{table}

%\title{5 Transmission eigenvalues closet to 4 with $|R|<1\subset |r|<0.5$. }

\subsection{Exterior scattering problem with the idealised boundary condition on $\partial D$}

In this section, we shall numerical demonstrate and verify Theorem~\ref{thm:2} with the idealised homogenous Dirichlet or Neumann condition on the boundary of the cloaked region $\partial D$. We first discuss the numerical strategy in determining the Herglotz approximation of the interior transmission eigenfunction $w$ in \eqref{eq:iten1} by $w_\kappa^g$ in \eqref{eq:app1}. In the sequel, we let $\Gamma'$ be a closed curve, and we require
\begin{equation}
w_\kappa^g|_{\Gamma'} \approx w|_{\Gamma'}.
\end{equation}
%Both $w$ and $w_\kappa^g$ satisfy the Helmholtz equation and they are equal on a closed curve, they should be the same on the whole space.
Assume $(\omega_i,\xi_i)$ is a quadrature rule for $\mathbb{S}^2$, and we choose the sampling points $x_1,\cdots,x_N$ on $\Gamma'$ such that
\begin{equation}\label{eq:ddd1}
\sum_i e^{i\kappa x_j\cdot \xi_i}g(\xi_i)\omega_i=w(x_j),\quad j=1,\cdots,N.
\end{equation}
\eqref{eq:ddd1} can be written as
\begin{equation}
\mathcal{A}g=W,
\end{equation}
which is ill-conditioned and we shall make use of the standard regularisation strategy by minimising the following Tikhonov functional
\begin{equation}\label{eq:ddd2}
||\mathcal{A}g-W||^2+r ||g||^2,
\end{equation}
where $r\in\mathbb{R}_+$ signifies a regulariser and should be properly chosen. In order to determine a minimiser of \eqref{eq:ddd2}, we solve the following normal equation
\begin{equation}
(r I +\mathcal{A}^*\mathcal{A})g=\mathcal{A}^*W.
\end{equation}
After the numerical determination of the Herglotz wave function $w^g_\kappa$, we use it as the incident wave $u^i$ to check if the following exterior problem has a small scattered wave filed,
\begin{equation}\label{eq:ddd3}
\begin{cases}
\displaystyle{\Delta u(x)+\kappa^2 n_c(x) u(x)=0}\qquad & \mbox{for}\ \ x\in\Omega\backslash\overline{D},\smallskip\\
\displaystyle{\Delta u(x)+\kappa^2 u(x)=0}\qquad & \mbox{for}\ \ x\in\mathbb{R}^N\backslash\overline{\Omega},\smallskip\\
\mathcal{B} u(x)=0\qquad & \text{for}\ \ x\in\partial D,\smallskip\\
\mbox{$u-u^i$ satisfies the Sommerfeld}&\hspace*{-2mm} \mbox{radiation condition}.
\end{cases}
\end{equation}

We make use the PML (perfectly matched layer) technique to reduce the unbounded problem \eqref{eq:ddd3} for the scatted wave $u^s=u-u^i$ to be a bounded-domain problem. The PML formulation takes the following form,
\begin{equation}\label{PML}
\frac{\partial}{\partial x^1}\Big(\frac{S_{x^2}}{S_{x^1}}u^s_{x^1}\Big)+\frac{\partial}{\partial x^2}\Big(\frac{S_{x^1}}{S_{x^2}}u^s_{x^2}\Big)+\kappa^2nS_{x^1}S_{x^2}u^s=-\kappa^2(n-1)u^i,
\end{equation}
where
\begin{equation}
S_{x^1}=1+\frac{\sigma_{x^1}}{\mathrm{i}\kappa},\quad S_{x^2}=1+\frac{\sigma_{x^2}}{\mathrm{i}\kappa},
\end{equation}
and $\sigma_{x^1}$ and $\sigma_{x^2}$ are, respectively, functions of $x^1, x^2$ only. In our numerical experiments, we choose
\begin{equation}
\sigma_{x^1}=(l/d)^m\sigma_{x,max},
\end{equation}
where $l$ is the distance from the interface between the PML and the physical solution domain, $d$ is the thickness of the PML, and
\begin{equation}
\sigma_{x^1,max}=-\frac{(m+1)\ln(R(0))}{2d},
\end{equation}
with $R(0)$ taken to be $e^{-16}$ and $m=3$.
The definition of $\sigma_{x^2}$ is similar. We define the scattering ratio as $\frac{||u^s||_{L^2(\Gamma)}}{||u^i||_{L^2(\Gamma)}}$ with $\Gamma$ being a closed curve chosen to be a circle of radius 1.8. In all the subsequent figures, we use the mesh with  $h=0.1$ and $n_c=16$.

Figure \ref{Fig1} presents the distributions of the scattered field, total field and incident field in the circle geometry when $\kappa=0.354349$ with the Dirichlet boundary condition. The scattering ratio is 0.032857. Figure \ref{Fig2} presents the fields in the circle geometry when $\kappa=3.028932$ with the Dirichlet boundary condition, and the scattering ratio is 0.014606. Figure \ref{Fig3} is the result for $\kappa=3.857263$ with the Dirichlet boundary condition and the scattering ratio is 0.061382. We also give a result in the circle geometry with the Neumann boundary condition in Figure \ref{Fig4}, where $\kappa=1.890939$  and the scattering ratio 0.032492. Clearly, the near-invisibility is achieved for all of the cases and this verifies Theorems~\ref{thm:2} and \ref{thm:3}. 

\begin{figure}[!ht]
\centerline{
   \includegraphics[height=5cm,width=12cm,angle=0]{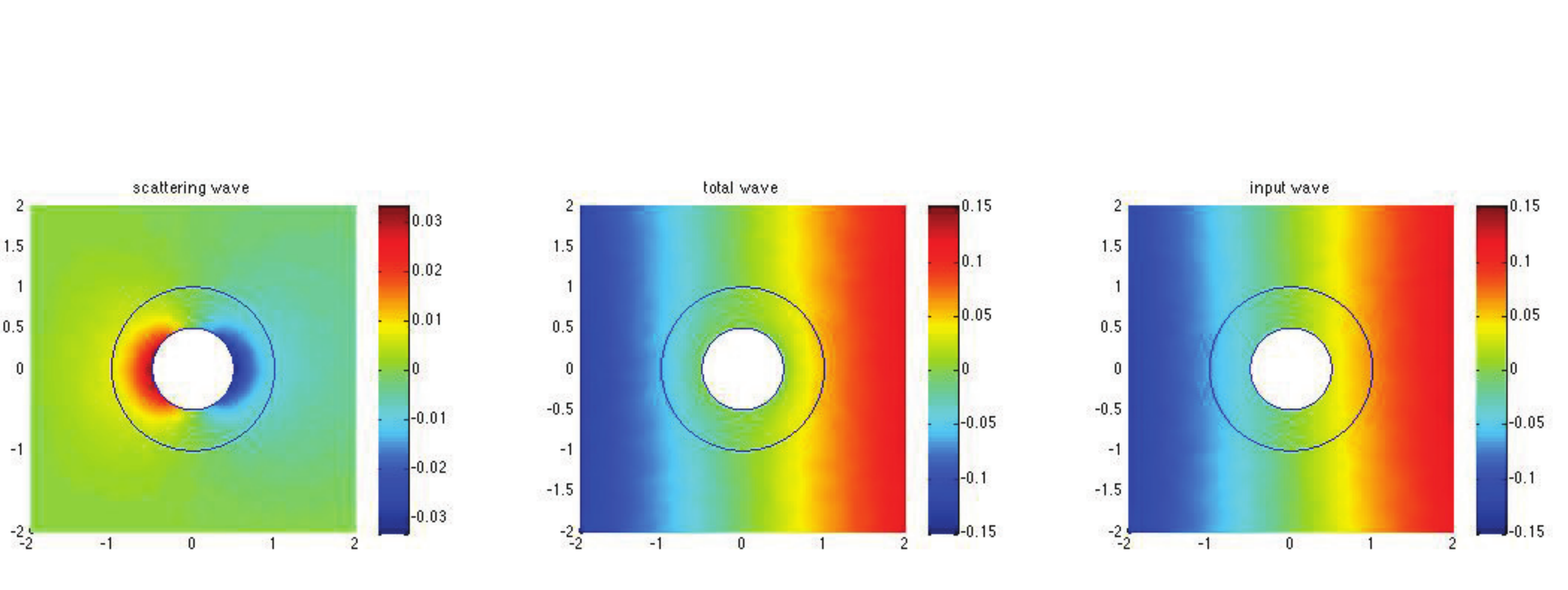}
  }
\caption{Field distributions in the circle geometry when $\kappa=0.354349$ with the Dirichlet boundary condition.}
\label{Fig1}
\end{figure}

\begin{figure}[!ht]
\centerline{
 \includegraphics[height=5cm,width=12cm,angle=0]{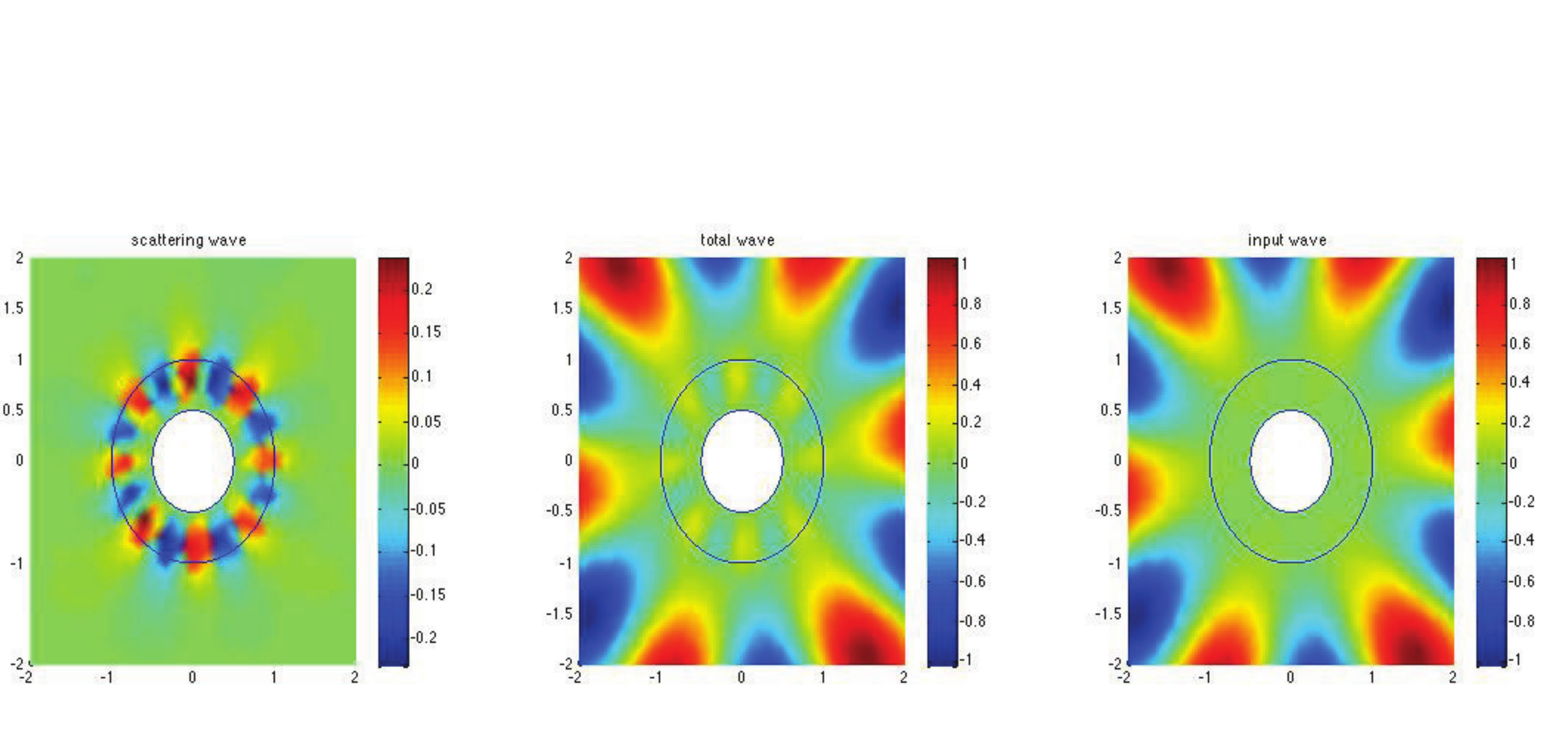}
  }
\caption{Field distributions in the circle geometry when $\kappa=3.028932$ with the Dirichlet boundary condition.}
\label{Fig2}
\end{figure}

\begin{figure}[!ht]
\centerline{
 \includegraphics[height=5cm,width=12cm,angle=0]{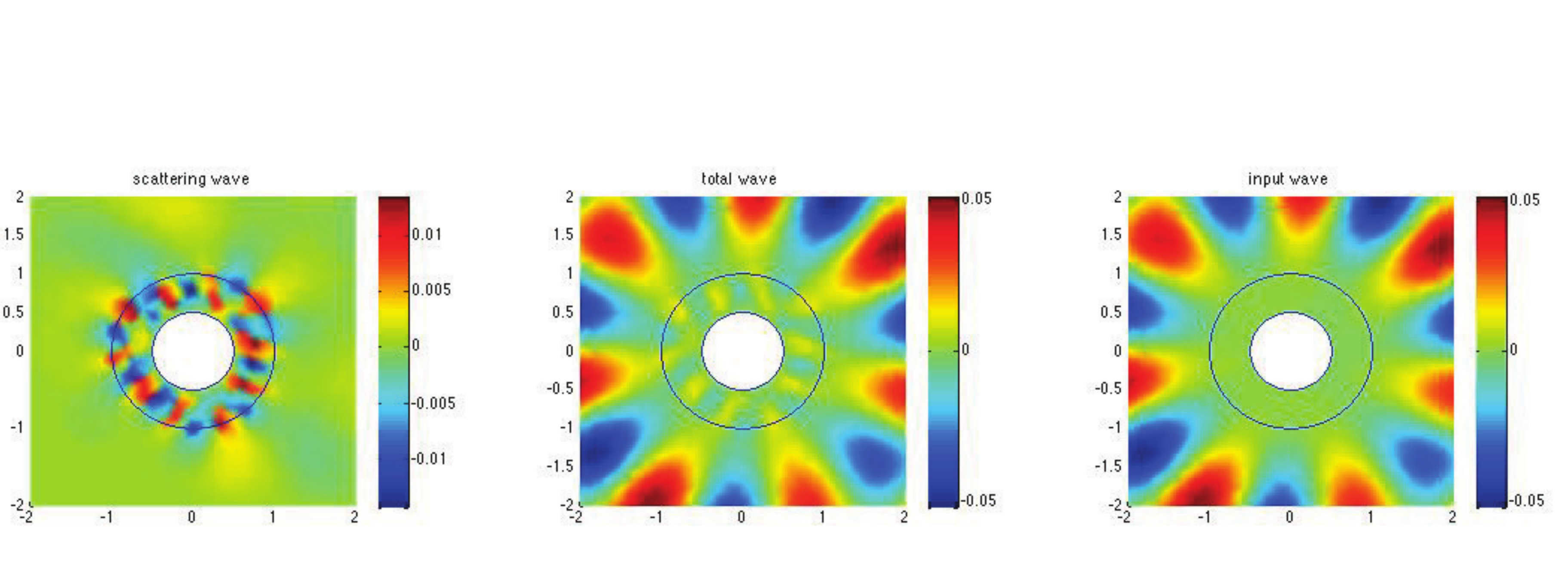}
  }
\caption{Field distributions in the circle geometry when $\kappa=3.857263$ with the Dirichlet boundary condition.}
\label{Fig3}
\end{figure}

\begin{figure}[!ht]
\centerline{
 \includegraphics[height=5cm,width=12cm,angle=0]{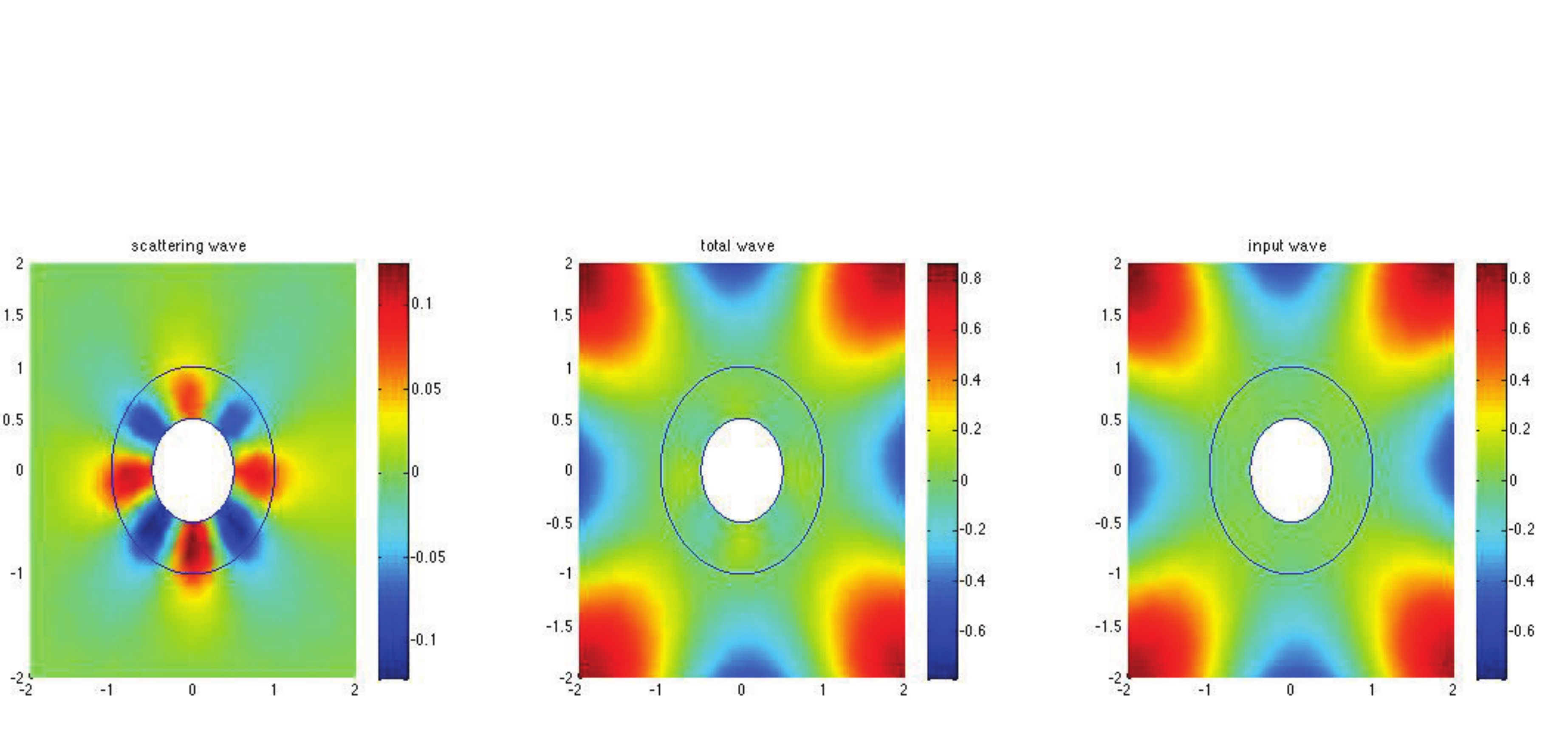}
  }
\caption{Field distributions in the circle geometry when $\kappa=1.890939$ with the Neumann boundary condition.}
\label{Fig4}
\end{figure}

Next, we will consider the examples for the ellipse geometry. Figure \ref{Fig5} gives the distributions of the scattered field, total field and incident field of the ellipse geometry when $\kappa=2.097681$ with the Dirichlet boundary condition, and the scattering ratio is 0.074928.  Figure \ref{Fig6} gives the results when $\kappa=1.747153$ with the Neumann boundary condition, and the scattering ratio is 0.061357. The near-invisibility is also achieved for all both of the computed cases. 

\begin{figure}[!ht]
\centerline{
 \includegraphics[height=5cm,width=12cm,angle=0]{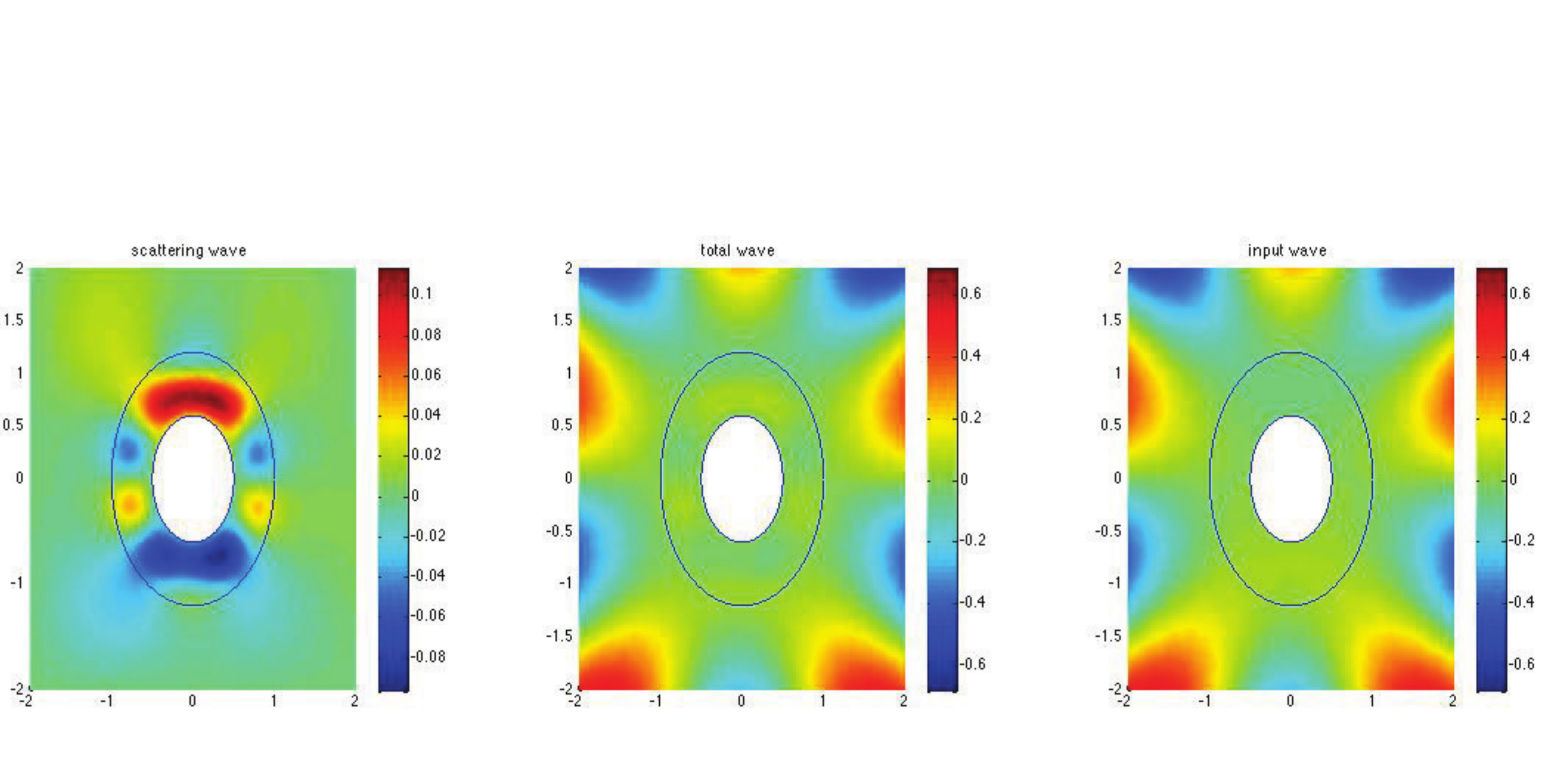}
  }
\caption{Field distributions in the ellipse geometry when $\kappa=2.097681$ with the Dirichlet boundary condition.}
\label{Fig5}
\end{figure}

\begin{figure}[!ht]
\centerline{
 \includegraphics[height=5cm,width=12cm,angle=0]{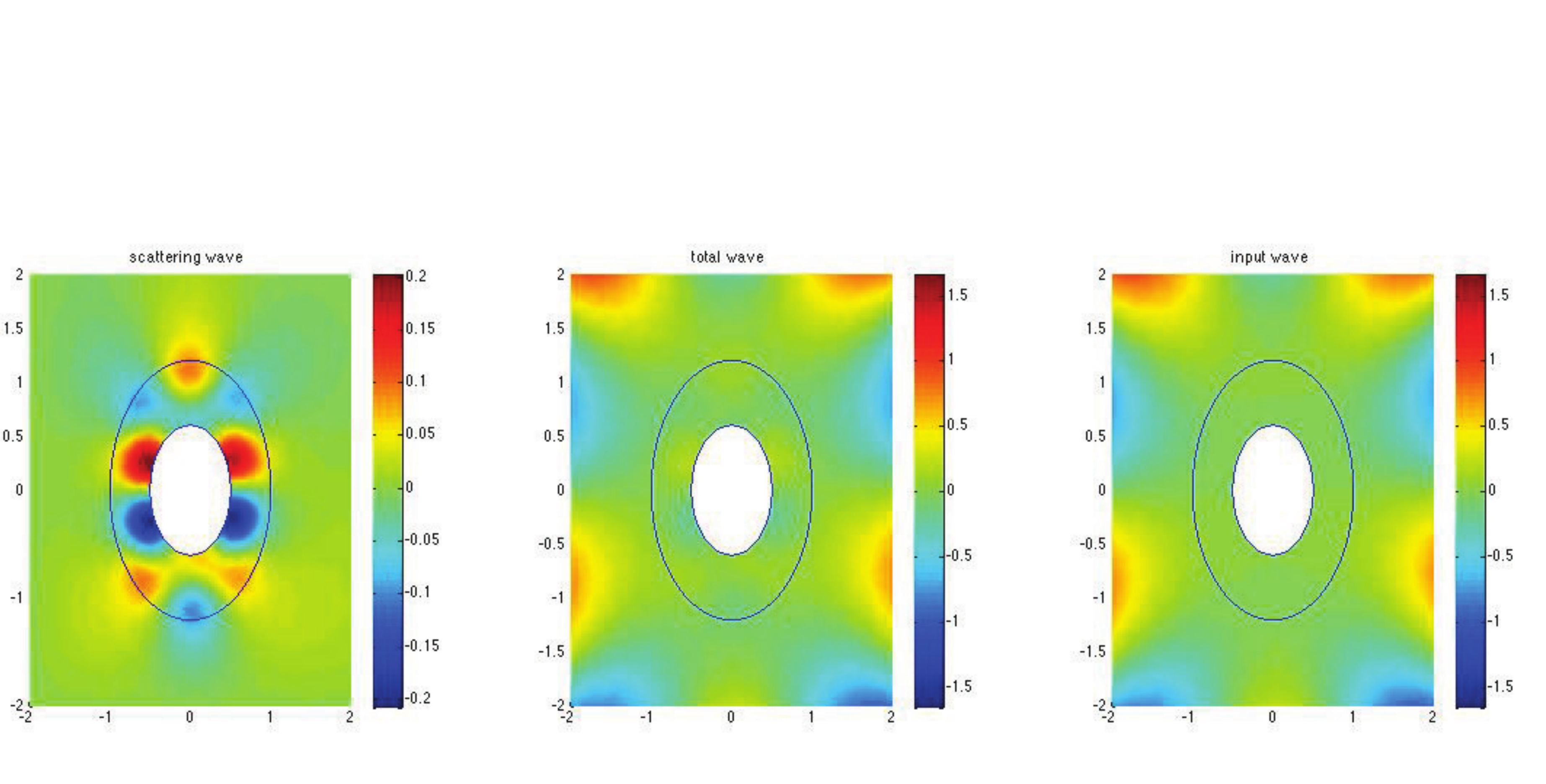}
  }
\caption{Field distributions in the ellipse geometry when $\kappa=1.747153$ with the Neumann boundary condition.}
\label{Fig6}
\end{figure}

Next, we consider the square geometry. Figure \ref{Fig7} presents the distributions of the scattered field, total field and incident field in the square geometry when $\kappa=2.431338$ with the Dirichlet boundary condition, and the scattering ratio is 0.816597. Figure \ref{Fig8} presents the corresponding results when $\kappa=0.761138$ with the Neumann boundary condition, and the scattering ratio is 1.136312. Clearly, near-invisibility cannot be achieved. We believe that this is mainly due to that the non-transparency condition is not fulfilled in those cases; see our remarks after Theorem~2.3 in Section 2. 

\begin{figure}[!ht]
\centerline{
 \includegraphics[height=5cm,width=12cm,angle=0]{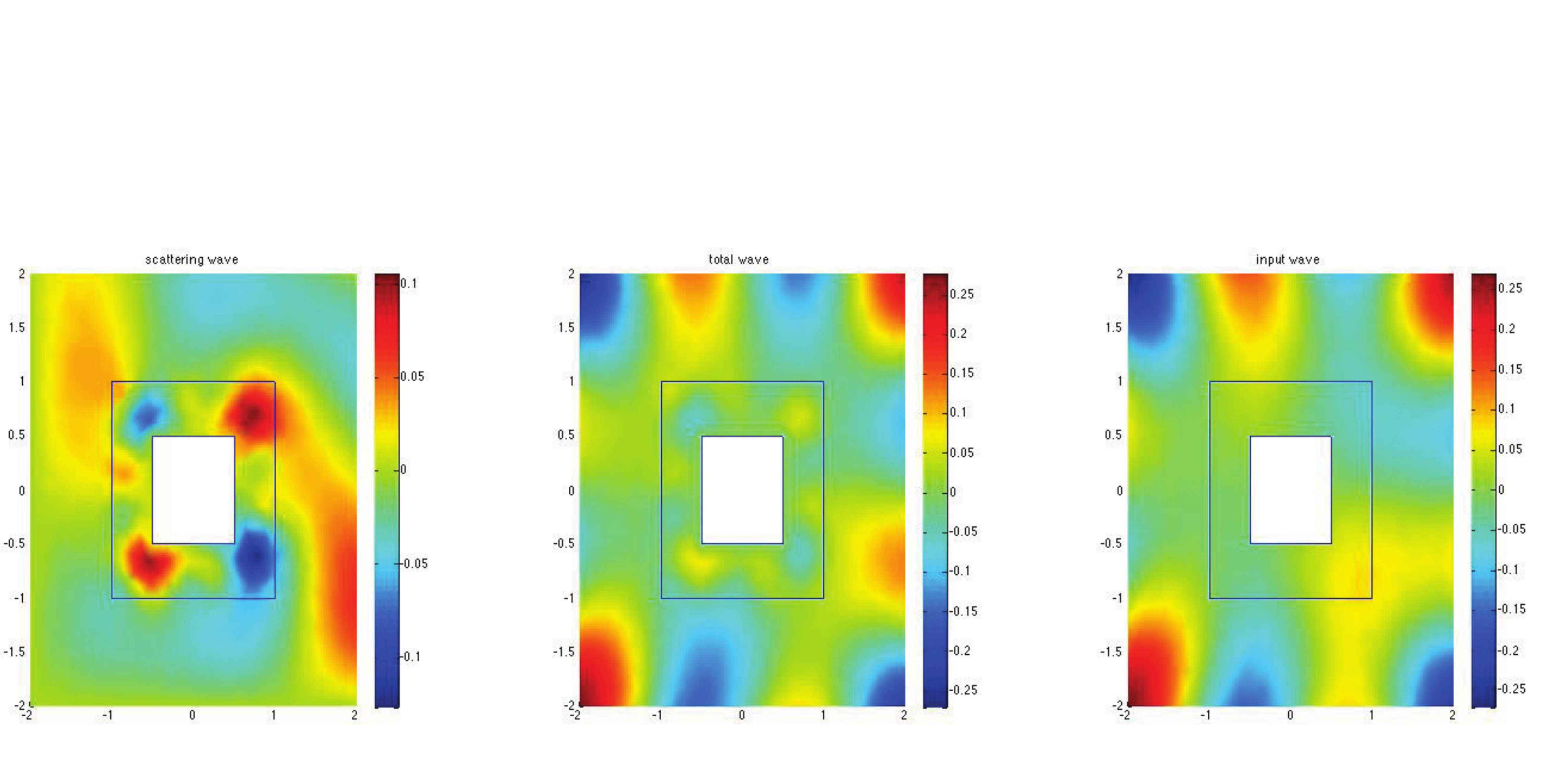}
  }
\caption{Field distributions in the square geometry when $\kappa=2.431338$ with the Dirichlet boundary condition.}
\label{Fig7}
\end{figure}

\begin{figure}[!ht]
\centerline{
 \includegraphics[height=5cm,width=12cm,angle=0]{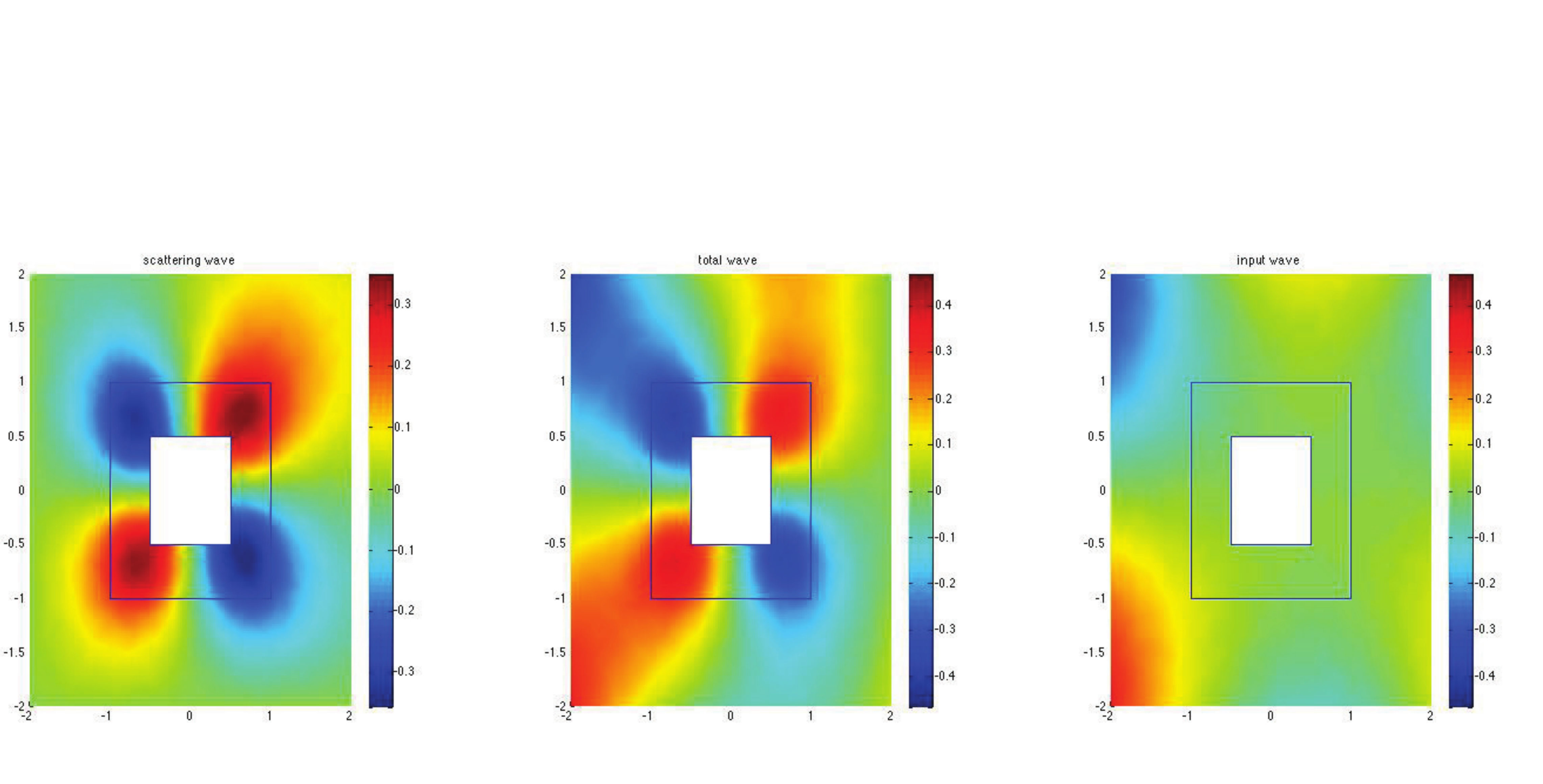}
  }
\caption{Field distributions in the square geometry when $\kappa=0.761138$ with the Neumann boundary condition.}
\label{Fig8}
\end{figure}

\subsection{Numerical results for isotropic invisibility cloaking}

In this section, we use the the isotropic lossy layers as in \eqref{eq:lossy1} and \eqref{eq:lossy2} for the finite realisation of the idealised boundary condition on $\partial D$. We define $\Sigma := \{\|x\|<0.3\} $ in the case of a circle and $\Sigma := \{ (\frac{x^1}{0.3})^2+(\frac{x^2}{0.36})^2<1\}$ for an ellipse; and $\sigma_a=I$ and $n_a=12$ in $\Sigma$. We also set $\gamma=1,\tau=0.01,\alpha=1,\beta=0.3$.

Figure \ref{Fig9} presents the distributions of the scattered field, total field and incident field in the circle geometry when $\kappa=3.857263$ with the setting \eqref{eq:lossy1}, and the scattering ratio is 0.100874. Figure \ref{Fig10} gives the field distributions when $\kappa=1.890939$ with the setting \eqref{eq:lossy2}, and the scattering ratio is 0.024135.

Figure \ref{Fig11} presents the distributions of the scattered field, total field and incident field in the ellipse geometry when $\kappa=2.097681$ with the setting \eqref{eq:lossy1}, and the scattering ratio is 0.247339. Figure \ref{Fig11} presents the distributions of the fields in the ellipse geometry when $\kappa=1.747153$ with the setting \eqref{eq:lossy2}, and the scattering ratio is 0.068411.

\begin{figure}[!ht]
\centerline{
 \includegraphics[height=5cm,width=12cm,angle=0]{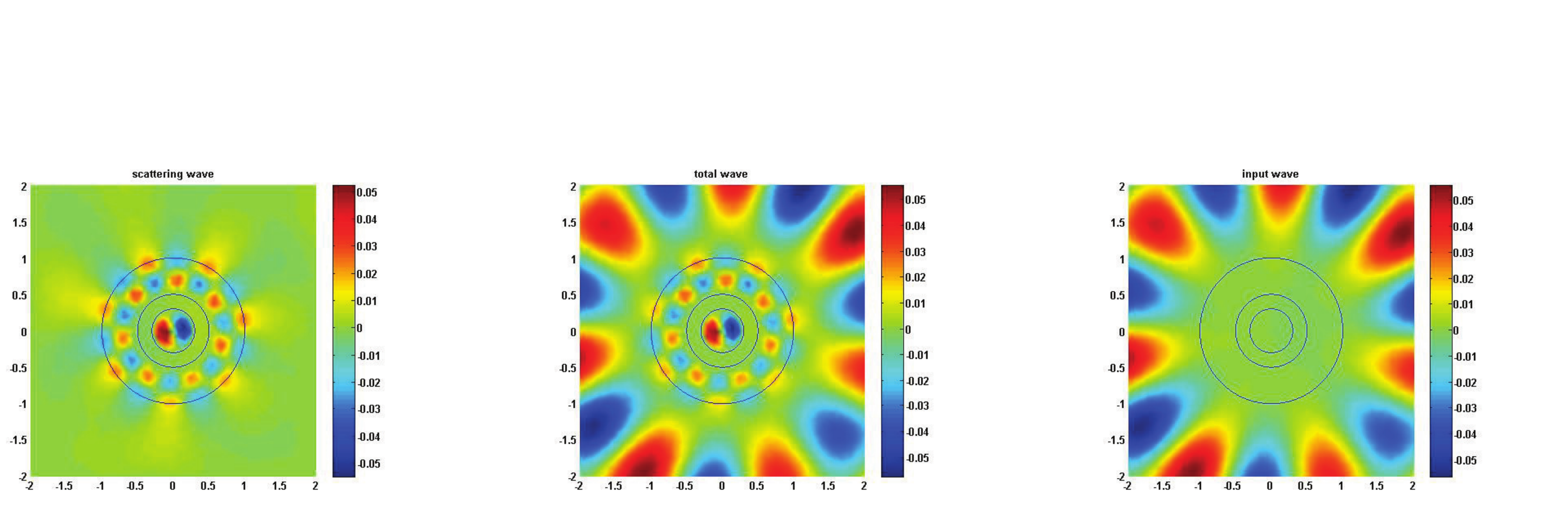}
  }
\caption{Field distributions in the circle geometry when $\kappa=3.857263$ with the setting \eqref{eq:lossy1}.}
\label{Fig9}
\end{figure}

\begin{figure}[!ht]
\centerline{
 \includegraphics[height=5cm,width=12cm,angle=0]{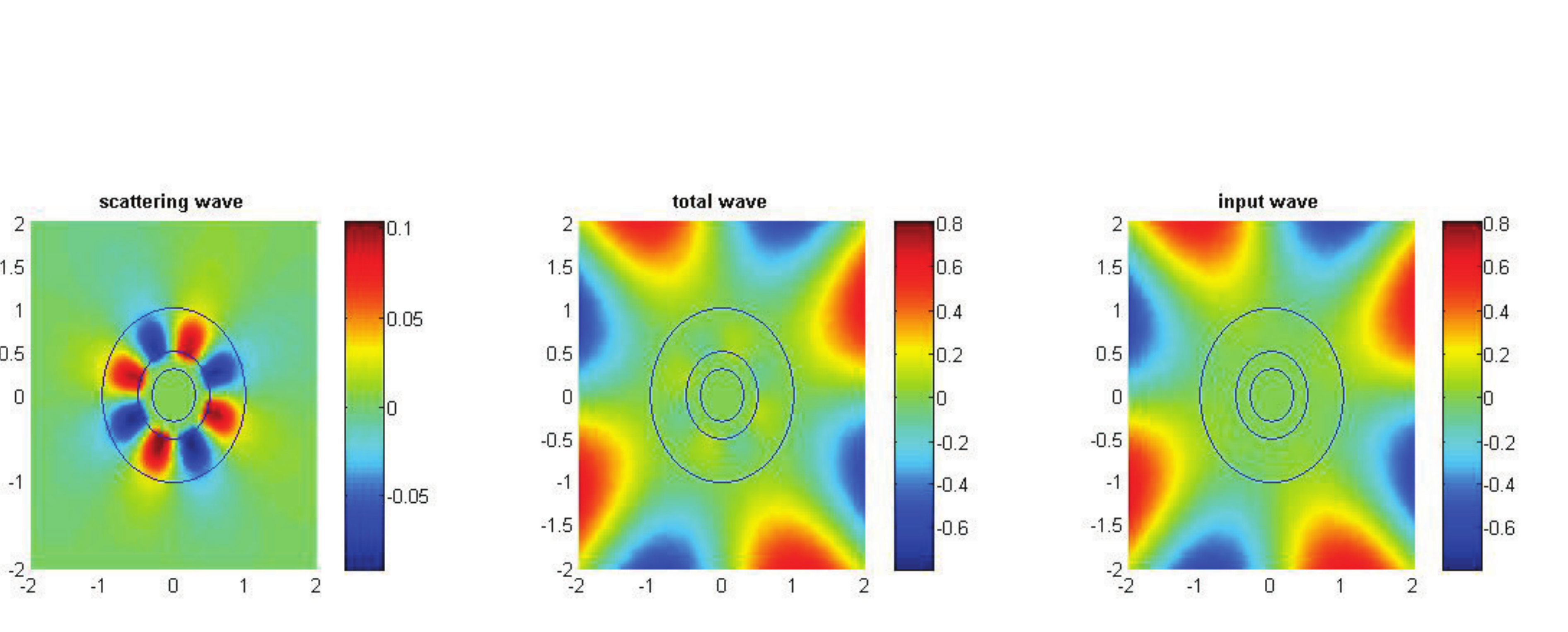}
  }
\caption{Field distributions in the circle geometry when $\kappa=1.890939$ with the setting \eqref{eq:lossy2}.}
\label{Fig10}
\end{figure}

\begin{figure}[!ht]
\centerline{
 \includegraphics[height=5cm,width=12cm,angle=0]{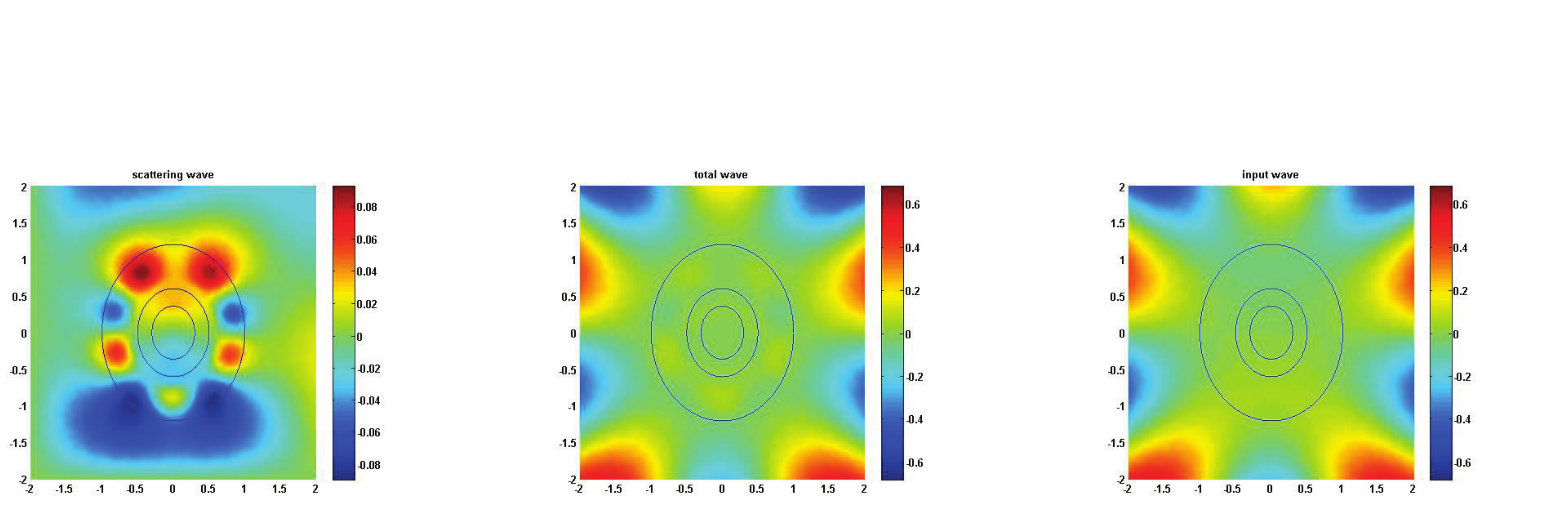}
  }
\caption{Field distributions in the ellipse geometry when $\kappa=2.097681$ with the setting \eqref{eq:lossy1}.}
\label{Fig11}
\end{figure}

\begin{figure}[!ht]
\centerline{
 \includegraphics[height=5cm,width=12cm,angle=0]{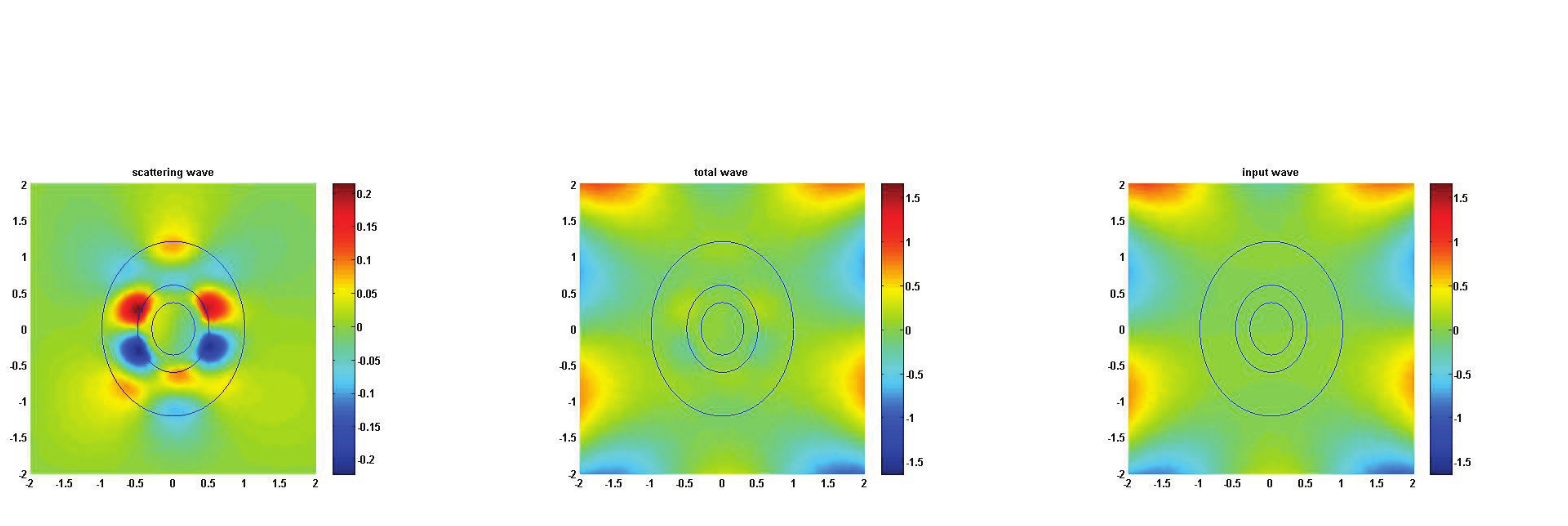}
  }
\caption{Field distributions in the ellipse geometry when $\kappa=1.747153$ with the setting \eqref{eq:lossy2}.}
\label{Fig12}
\end{figure}

\section*{Acknowledgement}

The work of H. Liu was supported by the FRG grants from Hong Kong Baptist University, Hong Kong RGC General Research Funds, 12302415 and 405513, and the NSFC grant under No. 11371115. X. Ji was supported by the National Natural Science Foundation of China (No. 11271018, No. 91230203) and the Special Funds for National Basic Research Program of China, 973 Program 2012CB025904.

\end{document}